\documentclass[a4paper,11pt,reqno]{amsart}

\usepackage[utf8]{inputenc}

\usepackage[a4paper,top=2cm,bottom=2cm,left=2.2cm,right=2.2cm]{geometry}

\usepackage{titlesec}
\usepackage{lipsum}
\usepackage{enumerate}
\usepackage{enumitem}
\usepackage{mathrsfs}
\usepackage{amsmath,amssymb,amsthm}
\usepackage{mathtools}
\usepackage{graphicx, float, subfigure}
\usepackage{url}
\usepackage{tikz}
\usepackage{bbold}
\usepackage{mleftright}
\usepackage{color}
\usepackage{xcolor}
\usepackage{colortbl}
\usepackage{array}

\usepackage{ytableau}
\usepackage[all,cmtip]{xy}
\usepackage{multirow}
\usepackage{makecell}
\usepackage[colorinlistoftodos]{todonotes}
\usepackage[colorlinks=true, allcolors=red]{hyperref}
\usepackage{harpoon}
\usepackage{MnSymbol}
\usepackage{esvect}
\usepackage{diagbox}
\usepackage[title]{appendix}
\usepackage{tikz-cd}

\theoremstyle{plain}
\newtheorem{theorem}{\scshape Theorem}[section]
\newtheorem{proposition}[theorem]{\scshape Proposition}
\newtheorem{lemma}[theorem]{\scshape Lemma}
\newtheorem{corollary}[theorem]{\scshape Corollary}
\newtheorem{conjecture}[theorem]{\scshape Conjecture}

\newtheorem*{assumption*}{\scshape Assumption}

\newtheorem*{claim*}{Claim}

\theoremstyle{definition}

\newtheorem{definition}[theorem]{\scshape Definition}
\newtheorem{remark}[theorem]{\scshape Remark}
\newtheorem{example}[theorem]{\scshape Example}

\newcommand{\Tam}{\mathsf{Tam}}
\renewcommand{\H}{\mathsf{H}}
\newcommand{\T}{\mathsf{T}}
\renewcommand{\Row}{\mathsf{Row}}
\newcommand{\alt}{\mathsf{alt}}
\newcommand{\Pop}{\mathsf{Pop}}

\numberwithin{equation}{section}

\titleformat{\section}{\centering\bfseries}{\thesection}{1em}{\MakeUppercase}
\titleformat{\subsection}{\bfseries}{\thesubsection}{1em}{}

\begin{document}
\title{Rowmotion on hook and two-row alt $\nu$-Tamari lattices}

\author{Sen-Peng Eu, Vei-Cheng Hioe, and Yi-Lin Lee}
\address{(S.-P. Eu) Department of Mathematics, National Taiwan Normal University, Taipei, Taiwan}
\email{speu@math.ntnu.edu.tw}

\address{(V.-C. Hioe) Department of Mathematics, National Taiwan Normal University, Taipei, Taiwan}
\email{Victor.hioe0829@gmail.com}

\address{(Y.-L. Lee) Department of Mathematics, National Taiwan Normal University, Taipei, Taiwan}
\email{yillee@ntnu.edu.tw}


\begin{abstract}
    In 2024, Ceballos and Chenevi{\`e}re introduced alt $\nu$-Tamari lattices, parameterized by a lattice path $\nu$ and an increment vector $\delta$, as a common generalization of $\nu$-Tamari and $\nu$-Dyck lattices. We study rowmotion on two families: the alt hook-Tamari lattice $\mathsf{H}_{\delta}(a,b)$ (where $\nu=EN^{a-1}E^{b-1}N$) and the alt $2$-row-Tamari lattice $\mathsf{T}_{\delta}(a,b)$ (where $\nu=E^aNE^bN$).

    We explicitly determine the orbit structures of $\mathsf{H}_{\delta}(a,b)$ and $\mathsf{T}_{\delta}(a,b)$ under rowmotion, and prove that their orbit structures are independent of the increment vector $\delta$. As a consequence, we show that rowmotion on $\mathsf{H}_{\delta}(a,b)$ exhibits the cyclic sieving phenomenon. We also compute orbit sums for several natural statistics. In the hook case, we evaluate the down-degree, peak, valley, and area statistics; in the $2$-row case, we focus on the down-degree statistic. All of these---except for the area statistic---are homometric under rowmotion.

    Regarding the methodology of this paper, our results in the hook case are obtained by applying a simple local modification to their Hasse diagrams. In the $2$-row case, we introduce a switching property for semidistributive lattices, which allows us to compare the orbit structures arising from different increment vectors.
\end{abstract}

\maketitle

\section{Introduction}\label{sec.intro}

The rapidly developing field of dynamical algebraic combinatorics aims to study the structure of operators acting on objects in algebraic combinatorics. One of the most extensively studied operators is \emph{rowmotion} on various important objects, such as Boolean lattices, chains, root posets, minuscule posets, Tamari lattices, and Cambrian lattices; see \cite[Section 7]{TW19} for a comprehensive review. Recently, rowmotion on $321$-avoiding permutations \cite{AE23}, fence posets \cite{EPRS23}, rooted trees \cite{DKLLSS23}, $m$-Tamari lattices and biCambrian lattices \cite{DL24} have been explored, and these works produced many interesting results.

Rowmotion is an operator, denoted by $\Row$, acting on the set of order ideals $\mathsf{J}(P)$ of a finite poset $P$. According to the fundamental theorem of finite distributive lattices, $\mathsf{J}(P)$ is isomorphic to a finite distributive lattice when ordered by inclusion (see \cite[Theorem 3.4.1]{ECI} for example). The concept of rowmotion has since been extended to various broader classes of lattices. In 2019, Barnard \cite{Barnard19} defined rowmotion on \textit{semidistributive lattices}, while Thomas and Williams \cite{TW19} introduced a version for \textit{trim lattices}, both of which generalize the distributive case. In 2023, Defant and Williams \cite{DW23} unified and further generalized these definitions to the class of \textit{semidistrim lattices}. 

The \textit{Tamari lattice}, a central object in algebraic combinatorics, provides a partial ordering on the set of Dyck paths. Since its introduction by Tamari \cite{Tamari62}, this lattice has received significant attention, and it has been generalized from various mathematical perspectives. In 2012, Bergeron and Préville-Ratelle \cite{BPR12} introduced \textit{$m$-Tamari lattices} to provide conjectural combinatorial interpretations for the dimensions of certain trivariate diagonal harmonic spaces. Subsequently, in 2017, Préville-Ratelle and Viennot \cite{PRV17} extended these to \textit{$\nu$-Tamari lattices} $\Tam(\nu)$, where $\nu$ is a given lattice path on the square lattice that starts at the origin and consists of unit north steps $N$ and unit east steps $E$. Throughout the paper, we may identify such a path with a finite word over the alphabet $\{N,E\}$, and use exponents to denote repeated concatenation of a word with itself; see Figure \ref{fig.expaths} for examples. The $n$th \textit{$m$-Tamari lattice} is a special case of $\nu$-Tamari lattices by taking $\nu = (NE^m)^n$; for $m=1$, this further reduces to the classical $n$th Tamari lattice.

Defant and Lin \cite{DL24} have investigated rowmotion on $m$-Tamari lattices. They fully determined the orbit structure of $m$-Tamari lattices under rowmotion, and proved that it exhibits the \textit{cyclic sieving phenomenon}. Additionally, they proved that the \textit{down-degree statistic} is \textit{homomesic} under rowmotion.

In 2024, Ceballos and Chenevi\`ere \cite{CC24} introduced a more general family of lattices, the \textit{alt $\nu$-Tamari lattices} $\Tam_{\delta}(\nu)$. These lattices consist of the same elements as $\nu$-Tamari lattices but carry different partial orderings depending on a vector $\delta$. They proved that the number of linear intervals in $\Tam_{\delta}(\nu)$ is independent of $\delta$. Building on a unified framework of \textit{framing lattices} by von Bell and Ceballos \cite{vBC24, vBC25}, it is now known that the various Tamari lattices mentioned above are all semidistributive. Consequently, it is a natural progression to study the behavior of rowmotion on these structures. Based on our data, we make the following conjecture. We would like to point out that this conjecture was independently\footnote{Private communication with Sam Hopkins on April 16, 2026.} formulated and proved by Adenbaum et al. in their forthcoming paper \cite{ABC+26+}.
\begin{conjecture}\label{conj}
    Let $\nu$ be a lattice path. The orbit structure of the alt $\nu$-Tamari lattice $\Tam_{\delta}(\nu)$ is independent of the increment vector $\delta$. 
\end{conjecture}

In this paper, we study rowmotion on two families of the alt $\nu$-Tamari lattices. Our results not only provide further evidence for Conjecture \ref{conj}, but also completely determine their orbit structures. 
\begin{itemize}
    \item[(1)] The lattice path $\nu = EN^{a-1}E^{b-1}N$, where $a$ and $b$ are positive integers. All the elements of $\Tam_{\delta}(\nu)$ are contained in the region bounded by $\nu$ and $N^{a}E^{b}$, that is, the Young diagram of the partition $(b,1^{a-1})$, known as the \textit{hook shape}. We refer to this structure as the \textit{alt hook-Tamari lattice} and denote it by $\H_{\delta}(a,b)$.

    \item[(2)] The lattice path $\nu = E^aNE^bN$, where $a$ and $b$ are nonnegative integers. All the elements of $\Tam_{\delta}(\nu)$ are contained in the region bounded by $\nu$ and $N^2E^{a+b}$, that is, the Young diagram of the partition $(a+b,a)$. We refer to this structure as the \textit{alt $2$-row-Tamari lattice} and denote it by $\T_{\delta}(a,b)$.
\end{itemize}
Figure \ref{fig.exhook} (resp., Figure \ref{fig.ex2row}) shows an element $\mu$ of the $\nu$-Tamari lattice where $\nu$ is the southeastern boundary of a hook shape (resp., $2$-row shape).
\begin{figure}[hbt!]
    \centering
    \subfigure[]{\label{fig.exhook}\includegraphics[width=0.25\textwidth]{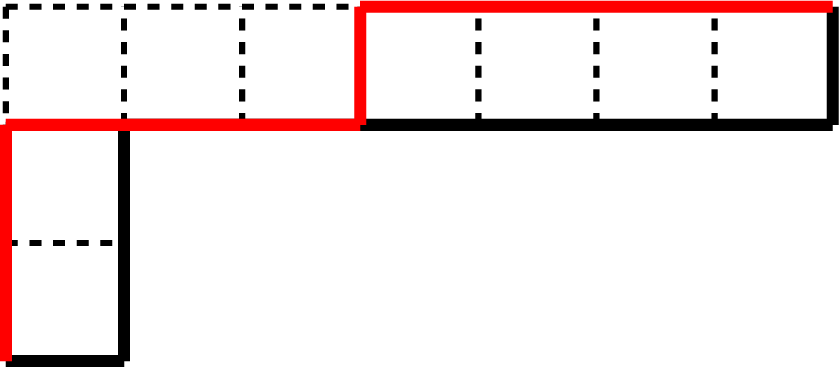}}
    \hspace{20mm}
    \subfigure[]{\label{fig.ex2row}\includegraphics[width=0.25\textwidth]{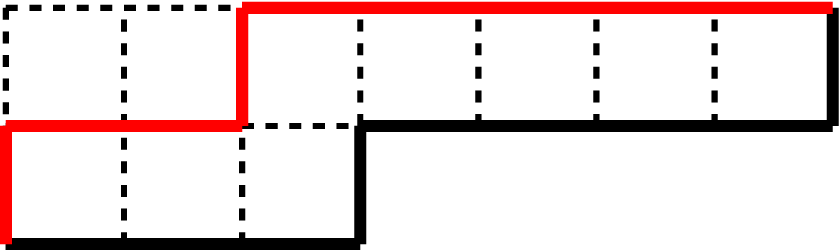}}
    \caption{(a) An element $\mu = N^2E^3NE^4$ (red) of the $\nu$-Tamari lattice with $\nu = EN^2E^6N$ (black). (b) An element $\mu = NE^2NE^5$ (red) of the $\nu$-Tamari lattice with $\nu = E^3NE^4N$ (black).}\label{fig.expaths}
\end{figure}

We state our main contributions as follows. First, we prove that the orbit structures of $\H_{\delta}(a,b)$ and $\T_{\delta}(a,b)$ under rowmotion are independent of the increment vector $\delta$; we also fully determine their orbit structures (Theorems \ref{thm.rowhook}, \ref{thm.2roworbit}, and \ref{thm.orbit}). Consequently, we show that rowmotion on $\H_{\delta}(a,b)$ exhibits the cyclic sieving phenomenon (Corollary \ref{cor.csp}). 

Second, we explicitly compute several statistics, including the down-degree, the number of peaks and valleys, and the area, over each orbit of $\H_{\delta}(a,b)$; we show that the first three of these statistics are homometric under rowmotion (Theorem \ref{thm.stathook}). For $\T_{\delta}(a,b)$, we compute the down-degree statistic and show that it is homometric under rowmotion (Theorem \ref{thm.orbit}). In the hook case, we notice that the rowmotion orbit structure coincides with that of the antichains of the fence poset with two segments, which was previously studied by Elizalde, Plante, Roby, and Sagan \cite{EPRS23}. Furthermore, our formulas for the down-degree and area statistics on $\H_{\delta}(a,b)$ coincide with the formulas for the statistics $\chi$ and $\hat{\chi}$ on fence posets introduced in \cite{EPRS23}; see Remark \ref{rmk}.

We rely on an equivalent expression of the elements of alt $\nu$-Tamari lattices, called \textit{$\nu$-bracket vectors}, to characterize the Hasse diagrams of $\H_{\delta}(a,b)$ and $\T_{\delta}(a,b)$. For rowmotion on the alt hook-Tamari lattices $\H_{\delta}(a,b)$, we apply a simple edge contraction on their Hasse diagrams to prove the invariance of orbit structures and determine their orbit structures explicitly. For rowmotion on the alt $2$-row-Tamari lattices $\T_{\delta}(a,b)$, we provide a novel switching property on semidistributive lattices (Theorem \ref{thm.switchrow}) to show the invariance of their orbit structures. We compute the orbit structures by relating them to linear congruence equations and solving these equations.

The rest of this paper is organized as follows. In Section \ref{sec.pre}, we provide the necessary background, including the definitions of alt $\nu$-Tamari lattices, rowmotion on semidistributive lattices, and several concepts from dynamical algebraic combinatorics. Section \ref{sec.hook} presents our investigations of rowmotion on the alt hook-Tamari lattices $\H_{\delta}(a,b)$, while Section \ref{sec.2row} presents our studies of rowmotion on the alt 2-row-Tamari lattices $\T_{\delta}(a,b)$. Finally, Section \ref{sec.remarks} contains concluding remarks and further discussion.

\section{Preliminaries}\label{sec.pre}

We assume a basic familiarity with poset theory, as presented in \cite[Chapter 3]{ECI}. For a lattice $(L, \leq)$ considered in this paper, $L$ is finite, and we embed its Hasse diagram into a rectangular coordinate system. Each element $x \in L$ is identified with a specific lattice point $(x_1, x_2)_L$; for brevity, we omit the subscript $L$ when the context is clear. For any elements $x=(x_1, x_2)$ and $y=(y_1, y_2)$ in $L$, we have $x \leq y$ if and only if $x_1 \leq y_1$ and $x_2 \leq y_2$. We also assume that a lattice path $\nu$ starts at the origin. 

In this section, we introduce some variants of $\nu$-Tamari lattices (Section \ref{sec.tam}), the $\nu$-bracket vectors (Section \ref{sec.bracket}), and rowmotion on semidistributive lattices (Section \ref{sec.row}). We also introduce the concept of cyclic sieving phenomenon (Section \ref{sec.csp}), homomesy, and homometry (Section \ref{sec.homomesy}) that will be discussed in this paper.

\subsection{The $\nu$-Dyck, $\nu$-Tamari, and alt $\nu$-Tamari lattices}\label{sec.tam}

    Let $\nu$ be a lattice path on the square lattice. A \textit{$\nu$-path} $\mu$ is a lattice path consisting of $N$ and $E$ steps that shares the same endpoints as $\nu$ and stays weakly above $\nu$. Let $P(\nu)$ denote the collection of all $\nu$-paths. By equipping $P(\nu)$ with different partial orders, one obtains the $\nu$-Dyck lattice, the $\nu$-Tamari lattice, and their common generalization, the alt $\nu$-Tamari lattices. These structures are defined as follows.
    
    The $\nu$-Dyck lattice, denoted by $\mathsf{Dyck}(\nu)$, is the poset on $P(\nu)$ where $p \leq q$ if $q$ stays weakly above $p$. The cover relation $p \lessdot q$ in $\mathsf{Dyck}(\nu)$ corresponds to the operation of replacing a valley $EN$ with a peak $NE$ at some specific position in the $\nu$-path $p$. In the specific case where $\nu = (NE^m)^n$, $\mathsf{Dyck}(\nu)$ reduces to the Dyck lattice on the set of $m$-Dyck paths. Moreover, when $m=1$, $\mathsf{Dyck}(\nu)$ reduces further to the classical Dyck lattice on the set of usual Dyck paths. The Hasse diagram of $\mathsf{Dyck}(\nu)$ with $\nu=EN^2E^2N$ is shown in Figure \ref{fig.alt0}.

    The $\nu$-Tamari lattice, denoted by $\Tam(\nu)$, was introduced by Préville-Ratelle and Viennot \cite{PRV17} using the set of $\nu$-paths $P(\nu)$. For any lattice point $x$ on a path $\mu \in P(\nu)$, define its \textit{altitude}\footnote{It is also called the \textit{horizontal distance} in \cite{PRV17}.} $\alt(\nu,x)$ to be the maximum number of $E$ steps that can be added to the right of $x$ without crossing $\nu$. Consider a valley $EN$ of $\mu$, and let $x$ be the lattice point between these $E$ and $N$ steps. Let $y$ be the first lattice point following $x$ such that $\alt(\nu,y) = \alt(\nu,x)$, and let $\mu(x,y)$ be the subpath of $\mu$ from $x$ to $y$. We define $\mu^{\prime}$ to be the path obtained from $\mu$ by interchanging the subpath $\mu(x,y)$ with the $E$ step immediately preceding $x$. This operation, denoted by $\mu \lessdot \mu^{\prime}$, is called the \textit{$\nu$-rotation} of $\mu$ at $x$. The $\nu$-Tamari lattice $\Tam(\nu)$ is the poset on $P(\nu)$ whose cover relations are given by these $\nu$-rotations; it was shown in \cite{PRV17} that $\Tam(\nu)$ is a lattice. The Hasse diagram of $\Tam(\nu)$ with $\nu=EN^2E^2N$ is shown in Figure \ref{fig.alt2}. Similarly, when specializing $\nu = (NE^m)^n$, $\Tam(\nu)$ is the $m$-Tamari lattices on the set of $m$-Dyck paths; when $m=1$, $\Tam(\nu)$ reduces to the classical Tamari lattices on the set of Dyck paths. 
    
    Alternative characterizations of $\Tam(\nu)$ using $\nu$-trees and $\nu$-bracket vectors were presented by Ceballos, Padrol, and Sarmiento \cite{CPS20}. A primary tool in our work is the interpretation of $\Tam(\nu)$ via $\nu$-bracket vectors, which will be discussed in detail in Section \ref{sec.bracket}.

    A common generalization of the $\nu$-Dyck and $\nu$-Tamari lattices was introduced by Ceballos and Chenevi\`ere \cite{CC24}, which is called the alt $\nu$-Tamari lattices. For a lattice path $\nu$, we may encode $\nu$ as a sequence of nonnegative integers $(\nu_0,\nu_1,\dots,\nu_n)$, where $n$ is the number of $N$ steps in $\nu$, $\nu_0$ is the number of initial $E$ steps, and $\nu_i$ is the number of consecutive $E$ steps immediately following the $i$th $N$ step in $\nu$. Note that the sum of this sequence $\nu_0+\cdots+\nu_n$ is the total number of $E$ steps in $\nu$. 

    We say $\delta=(\delta_1,\dots,\delta_n)$ is an \textit{increment vector} of $\nu$ if $\delta_i$ is an integer with $0 \leq \delta_i \leq \nu_i$ for $i=1,\dots,n$. Note that the vector $\delta$ begins with the index $1$ while the sequence $\nu$ begins with the index $0$. For any lattice point $x$ on a path $\mu \in P(\nu)$, define its \textit{$\delta$-altitude} $\alt_{\delta}(\nu,x)$ recursively as follows. 
    \begin{equation}\label{eq.alt}
        \alt_{\delta}(\nu,x) = \begin{cases}
            0, & \text{ if $x$ is the initial point of $\mu$,}\\
            \alt_{\delta}(\nu,w)-1, & \text{ if $\mu(w,x)$ is an $E$ step in $\mu$,}\\
            \alt_{\delta}(\nu,w)+\delta_i, & \text{ if $\mu(w,x)$ is the $i$th $N$ step in $\mu$,}
        \end{cases}
    \end{equation}
    where $w$ is the lattice point immediately preceding $x$ in $\mu$ and $\mu(w,x)$ is the single step from $w$ to $x$. 
    
    Consider a valley $EN$ of $\mu$, and let $x$ be the lattice point between these $E$ and $N$ steps. Let $y$ be the first lattice point following $x$ such that $\alt_{\delta}(\nu,y) = \alt_{\delta}(\nu,x)$, and let $\mu(x,y)$ be the subpath of $\mu$ from $x$ to $y$. We define $\mu^{\prime}$ to be the path obtained from $\mu$ by interchanging the subpath $\mu(x,y)$ with the $E$ step immediately preceding $x$. This operation, denoted by $\mu \lessdot \mu^{\prime}$, is called the \textit{$\delta$-rotation} of $\mu$ at $x$.

    Given a lattice path $\nu$, let $\delta$ be an increment vector of $\nu$. The alt $\nu$-Tamari lattice $\Tam_{\delta}(\nu)$ is a poset on $P(\nu)$ whose cover relations are given by these $\delta$-rotations; it was shown in \cite[Corollary 4.5]{CC24} that $\Tam_{\delta}(\nu)$ is a lattice. See Figure \ref{fig.alttam} for examples. 

\begin{figure}[htbp]
    \centering
    \subfigure[]{\label{fig.alt0}\includegraphics[width=0.4\textwidth]{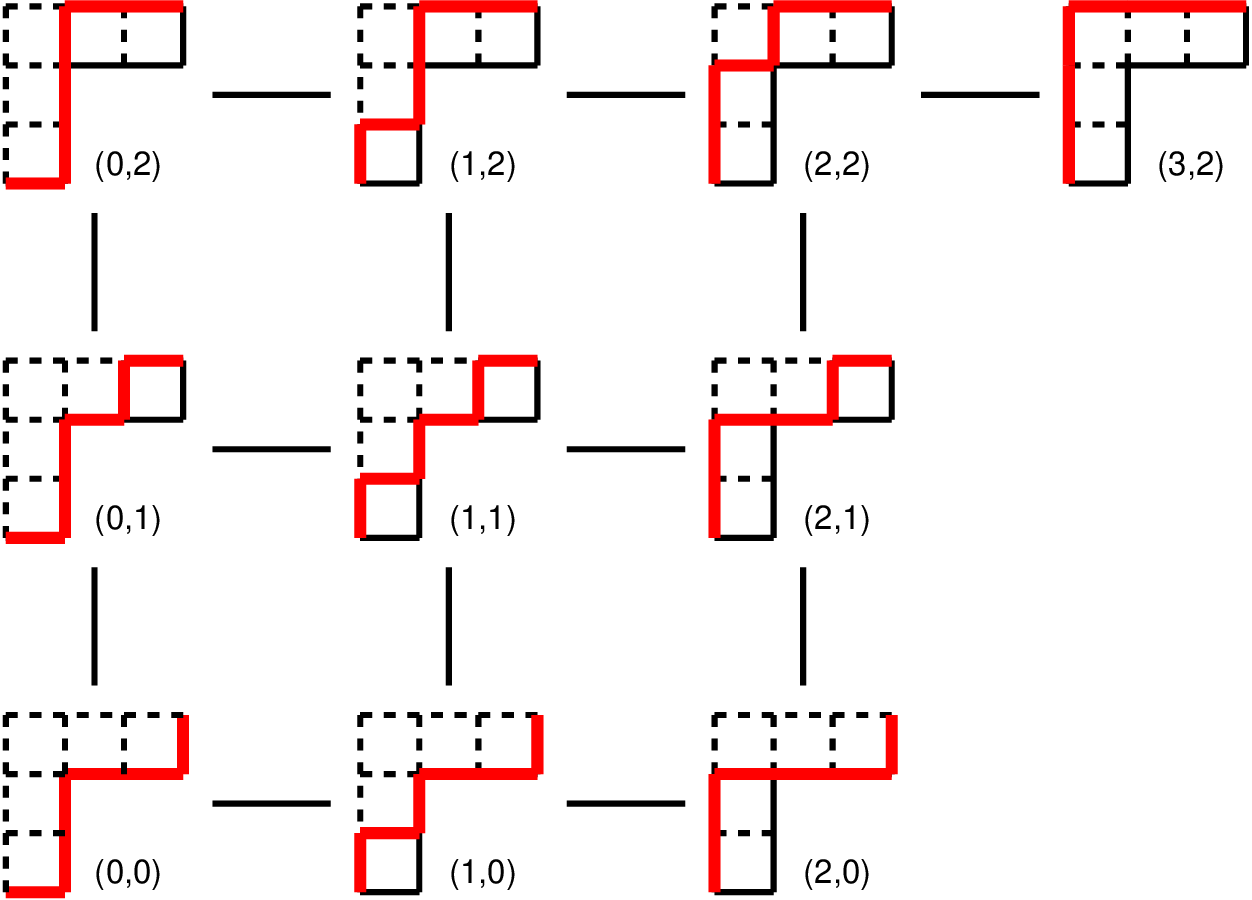}}
    \hspace{15mm}
    \subfigure[]{\label{fig.alt1}\includegraphics[width=0.4\textwidth]{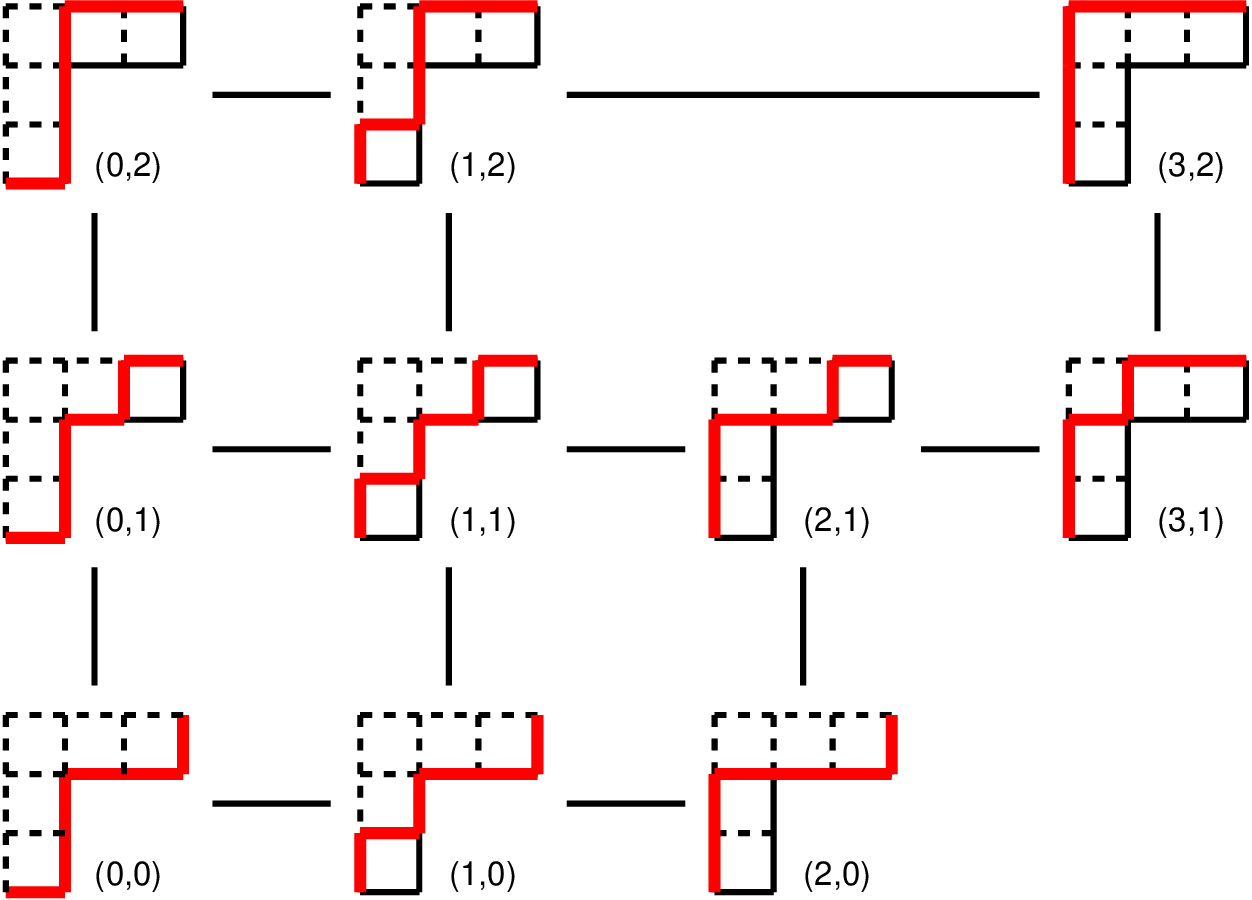}}
    \bigskip
    \vspace{2mm}
    \subfigure[]{\label{fig.alt2}\includegraphics[width=0.4\textwidth]{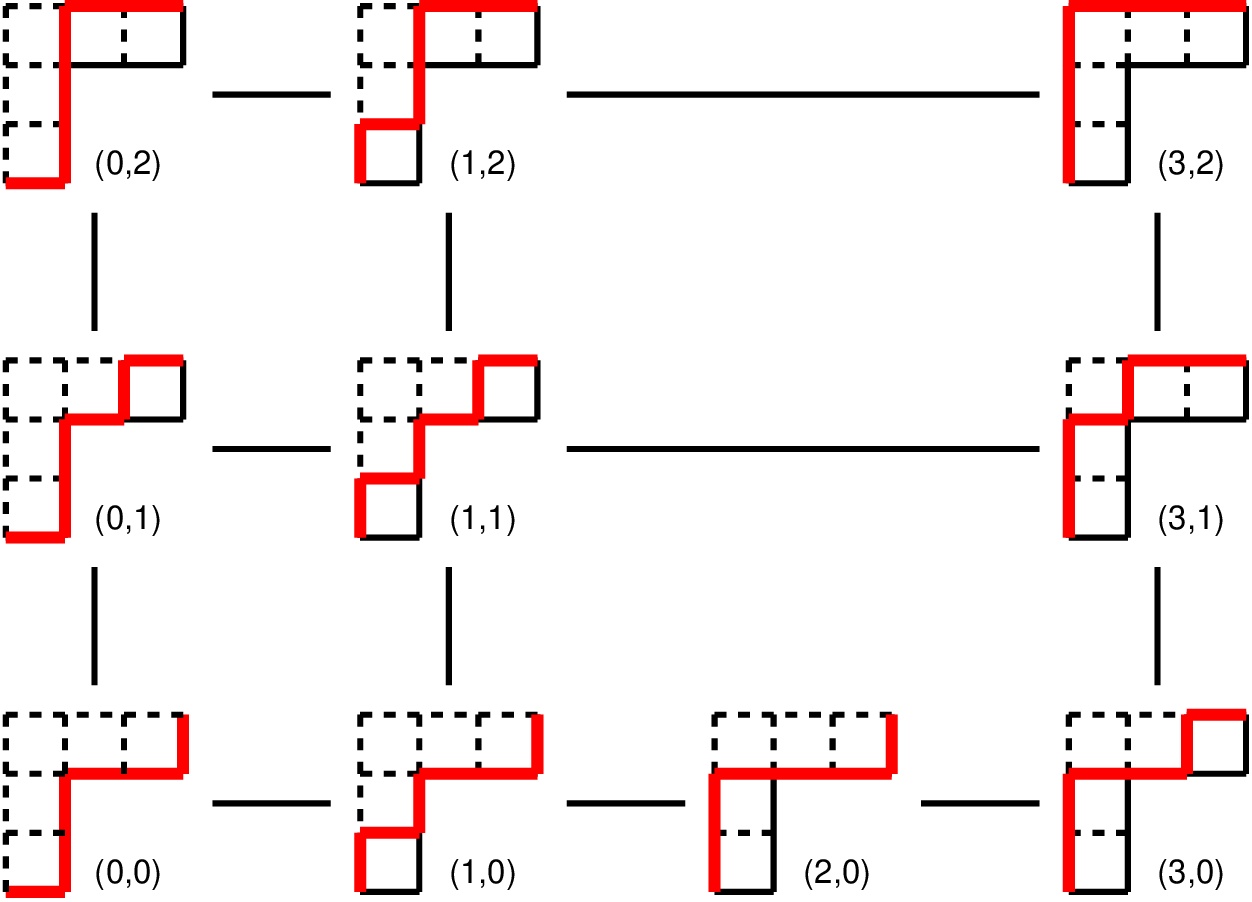}}
    \caption{The Hasse diagram of the alt hook-Tamari lattice $\Tam_{\delta}(\nu)$, where $\nu=EN^2E^2N$ and (a) $\delta=(0,0)$, (b) $\delta=(1,0)$, and (c) $\delta=(2,0)$. The coordinate of each element is also given. }\label{fig.alttam}
\end{figure}

    As pointed out in the introduction, the alt $\nu$-Tamari lattices are a subfamily of framing lattices introduced in the work of von Bell and Ceballos \cite{vBC24,vBC25}. They proved that the framing lattices have many nice properties---in particular---they are semidistributive (\cite[Theorem A]{vBC25}).
    
    Remarkably, there are two extremal choices of the increment vector $\delta$. If we take $\delta_i=\nu_i$ for all $i$, then the alt $\nu$-Tamari lattice coincides with the $\nu$-Tamari lattice. This follows from the fact (\cite[Remark 4.3]{CC24}) that $\alt_{\delta}(\nu,x) = \alt(\nu,x) - \nu_0$, hence the $\delta$-rotation is equivalent to the $\nu$-rotation. On the other hand, if we take $\delta_i=0$ for all $i$, then the alt $\nu$-Tamari lattice coincides with the $\nu$-Dyck lattice. This is due to an easy fact that the $\delta$-rotation of a path is equivalent to turning a valley on that path into a peak.
    
    We close this section by stating a useful property of the alt $\nu$-Tamari lattices. Define $\widehat{\nu}_0 = \sum_{i=0}^{n}\nu_i - \sum_{i=1}^{n}\delta_i$ and $\widehat{\nu}=(\widehat{\nu}_0,\delta_1,\dots,\delta_n)$. Notice that $\widehat{\nu}$ is a path below $\nu$ with the same endpoints as $\nu$. We write $1^{\nu} = N^n E^m$ for the top path above $\nu$ and $\widehat{\nu}$, where $m=\sum_{i=0}^n \nu_i=\sum_{i=0}^n \widehat{\nu}_i$. Recall that for two elements $a,c$ of a poset $P$, the \textit{interval} $[a,c]$ consists of all elements $b \in P$ such that $a \leq b \leq c$.

    \begin{proposition}[{\cite[Proposition 4.4]{CC24}}]\label{prop.int}
        Given a lattice path $\nu$ and an increment vector $\delta$ of $\nu$. Let $\widehat{\nu}$ be defined as above. Then the alt $\nu$-Tamari lattice $\Tam_{\delta}(\nu)$ is the interval $[\nu, 1^{\nu}]$ in the $\widehat{\nu}$-Tamari lattice $\Tam(\widehat{\nu})$. 
    \end{proposition}

\subsection{The $\nu$-bracket vectors}\label{sec.bracket}

     A very useful description of the elements of $\nu$-Tamari lattices, called $\nu$-bracket vectors, was presented in \cite{CPS20} (see also \cite[Section 3.2]{DL24}). By convention, the $\nu$-bracket vectors are $0$-indexed. We give an equivalent definition of these vectors as follows. 

     First, fix a lattice path $\nu$ that starts at $(0,0)$ and ends at $(r-n,n)$. We denote by $\mathbf{a}(\nu)=(a_0(\nu),\dots,a_{r}(\nu))$ the vector obtained by recording the $y$-coordinates of the lattice points along $\nu$ in the order in which they appear. For $0 \leq i \leq n$, let $f_i$ be the maximum index such that $a_{f_i}(\nu) = i$. The indices $f_0,f_1,\dots,f_n$ are called the \emph{fixed positions} of $\nu$. 

     Next, given a $\nu$-path $\mu$, the \emph{$\nu$-bracket vector of $\mu$} is a vector $\mathbf{b}(\mu)=(b_0(\mu),\dots,b_{r}(\mu))$ obtained from the following procedures:
    \begin{itemize}
        \item[(1)] Set $b_{f_i}(\mu) = i$ for each $i=0,1,\dots,n$.
        \item[(2)] Let $g_i$ be the number of $i$'s in the vector $\mathbf{a}(\mu)$. Starting with $i=0$, from right to left, we successively assign $i$ to the $g_i-1$ available spots on the left of the position $f_i$. Increase $i$ by $1$ and repeat (2) until $i=n+1$.
    \end{itemize}
    Clearly, the $\nu$-bracket vector of $\nu$ is given by $\mathbf{b}(\nu) = \mathbf{a}(\nu)$. On the set of $\nu$-bracket vectors, define the \textit{componentwise} partial order $\leq$ by the condition that $\mathbf{b}(\mu) \leq \mathbf{b}(\mu^{\prime})$ if and only if $b_i(\mu) \leq b_i(\mu^{\prime})$ for all $i$. In \cite{CPS20}, they proved that the set of all $\nu$-bracket vectors forms a lattice and it is isomorphic to the $\nu$-Tamari lattice. 
    
    \begin{example}
        We consider the Tamari lattice $\Tam(\nu)$ shown in Figure \ref{fig.alt2}, where the lattice path $\nu=EN^2E^2N$ goes from $(0,0)$ to $(3,3)$. The $\nu$-bracket vector of $\nu$ is given by $\mathbf{b}(\nu) = (0,\underline{0},\underline{1},2,2,\underline{2},\underline{3})$, the underlined numbers indicate the fixed positions of $\nu$. Take $\mu = NENE^2N$, then $\mathbf{a}(\mu)=(0,1,1,2,2,2,3)$ and its $\nu$-bracket vector is given by $\mathbf{b}(\mu) = (1,\underline{0},\underline{1},2,2,\underline{2},\underline{3})$. Clearly, $\nu \leq 
        \mu$ in $\Tam(\nu)$ if and only if $\mathbf{b}(\nu) \leq \mathbf{b}(\mu)$. 
    \end{example}

\subsection{Rowmotion on semidistributive lattices}\label{sec.row}

    A \textit{lattice} $L$ is a poset such that any two elements $x,y \in L$ have a unique greatest lower bound (i.e., the \textit{meet} of $x$ and $y$, denoted $x \wedge y$) and a unique least upper bound (i.e., the \textit{join} of $x$ and $y$, denoted $x \vee y$). Given a subset $X$ of $L$, we write $\bigwedge X$ and $\bigvee X$ for the meet and join of all elements of $X$, respectively. By convention, $\bigwedge \emptyset = \hat{1}$ and $\bigvee \emptyset = \hat{0}$. An element $x \in L$ is \textit{join-irreducible} if $x$ covers exactly one element; we denote by $x_{*}$ the unique element covered by $x$. Similarly, an element $y \in L$ is \textit{meet-irreducible} if $y$ is covered by exactly one element; we denote by $y^{*}$ the unique element that covers $y$. 

    A lattice $L$ is \textit{semidistributive} if for all $x,y,z \in L$,
    \begin{equation*}
        x \wedge y = x \wedge z \Rightarrow x \wedge(y \vee z) = x \wedge y \quad\text{and} \quad x \vee y = x \vee z \Rightarrow x \vee (y \wedge z) = x \vee y.
    \end{equation*}
    Equivalently, $L$ is semidistributive if and only if for all $x,y \in L$ with $x \leq y$, the set $\{z \in L | z \wedge y = x\}$ has a unique maximal element and the set $\{z \in L | z \vee x = y\}$ has a unique minimal element. In particular, every distributive lattice is semidistributive. 

    The original definition of rowmotion on semidistributive lattices was given by Barnard \cite{Barnard19}. Here, we present another way to describe rowmotion via the \textit{pop-stack sorting operator}, introduced by Defant and Williams \cite[Section 9]{DW23}. 

    Let $L$ be a lattice and $x \in L$. The pop-stack sorting operator $\Pop_L^{\downarrow} : L \rightarrow L$ and the dual pop-stack sorting operator $\Pop_L^{\uparrow} : L \rightarrow L$ are defined by
    \begin{equation}
        \Pop_L^{\downarrow}(x) = x \wedge \bigwedge \{ y \in L | y \lessdot x\} \quad \text{and} \quad \Pop_L^{\uparrow}(x) = x \vee \bigvee \{ y \in L | x \lessdot y\}. 
    \end{equation}
    Rowmotion on semidistributive lattices is characterized by the following theorem \cite[Theorem 9.3]{DW23}\footnote{In fact, this result holds for the broader class of semidistrim lattices, which contains semidistributive lattices as a subset (\cite[Theorem 6.2]{DW23}).}. For brevity, we omit the subscript $L$ when the context is clear.
    \begin{theorem}[\cite{DW23}]\label{thm.charrow}
        Let $L$ be a semidistributive lattice and $x \in L$. Then $\Row_L(x)$ is a maximal element of the set $\{y \in L | \Pop_L^{\downarrow}(x) = x \wedge y \}$, and $\Row_L^{-1}(x)$ is a minimal element of the set $\{y \in L | \Pop_L^{\uparrow}(x) = x \vee y \}$.
    \end{theorem}

    \begin{example}
        Let $L$ be a semidistributive lattice with $13$ elements shown in Figure \ref{fig.rowexample}. We demonstrate how rowmotion acts on $L$. For the maximal element $m$, $\Pop^{\downarrow}(m)=m \wedge \bigwedge\{i,\ell\} = g$ and the set $\{y \in L | \Pop^{\downarrow}(m) = m \wedge y\}$ contains only one element $g$, so $\Row(m)=g$. For the element $g$, one can check that $\Pop^{\downarrow}(g)=b$ and  $\Row(g)=b$. For the element $b$, we have $\Pop^{\downarrow}(b)=a$ and the set $\{y \in L | \Pop^{\downarrow}(b) = b \wedge y\}$ contains three elements $\{a,e,j\}$ where the maximal element is $j$. Thus, $\Row(b)=j$. Continuing this process, we obtain the orbit structure of rowmotion on $L$. We encourage the reader to check that, in fact, rowmotion on $L$ has only one orbit:
        \begin{equation*}
            m \rightarrow g \rightarrow b \rightarrow j \rightarrow i \rightarrow h \rightarrow c \rightarrow k \rightarrow e \rightarrow d \rightarrow \ell \rightarrow f \rightarrow a \rightarrow m.
        \end{equation*}
    \end{example}
\begin{figure}[hbt!]
    \centering
        \begin{tikzpicture}
            \draw (0,0) -- (1,0) -- (2,0) -- (3,0);
            \draw (0,1) -- (1,1) -- (2,1) -- (3,1) -- (4,1);
            \draw (0,2) -- (1,2) -- (2,2) -- (4,2);
            \draw (0,0) -- (0,1) -- (0,2);
            \draw (1,0) -- (1,1) -- (1,2);
            \draw (2,0) -- (2,1) -- (2,2);
            \draw (4,1) -- (4,2);
            \draw (3,0) -- (3,1);
            \filldraw[black] (0,0) circle (2pt) node [anchor=north east] {$a$};
            \filldraw[black] (1,0) circle (2pt) node [anchor=north east] {$b$};
            \filldraw[black] (2,0) circle (2pt) node [anchor=north east] {$c$};
            \filldraw[black] (3,0) circle (2pt) node [anchor=north east] {$d$};
            \filldraw[black] (0,1) circle (2pt) node [anchor=north east] {$e$};
            \filldraw[black] (1,1) circle (2pt) node [anchor=north east] {$f$};
            \filldraw[black] (2,1) circle (2pt) node [anchor=north east] {$g$};
            \filldraw[black] (3,1) circle (2pt) node [anchor=north east] {$h$};
            \filldraw[black] (4,1) circle (2pt) node [anchor=north east] {$i$};
            \filldraw[black] (0,2) circle (2pt) node [anchor=north east] {$j$};
            \filldraw[black] (1,2) circle (2pt) node [anchor=north east] {$k$};
            \filldraw[black] (2,2) circle (2pt) node [anchor=north east] {$\ell$};
            \filldraw[black] (4,2) circle (2pt) node [anchor=north east] {$m$};
        \end{tikzpicture}
    \caption{A semidistributive lattice $L$ with $13$ elements.}
    \label{fig.rowexample}
\end{figure}
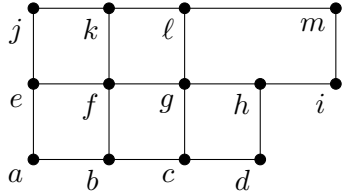

\subsection{The cyclic sieving phenomenon}\label{sec.csp}

    In their seminal paper, Reiner, Stanton, and White \cite{RSW04} introduced the cyclic sieving phenomenon (CSP), which provides a unified way to describe the orbit structure of a set of combinatorial objects under a cyclic action. Formally, let $X$ be a finite set, and $g:X \rightarrow X$ be an invertible map of order $n$, that is, $n$ is the smallest integer such that for all $x \in X$, $g^n(x)=x$. Let $f(q) \in \mathbb{C}[q]$ be a polynomial. We say the triple $(X,g,f(q))$ \textit{exhibits the cyclic sieving phenomenon} if for all integers $d$, 
    \begin{equation}
        |\{ x \in X | g^d(x) = x\}| = f(\omega^d),
    \end{equation}
    where $\omega = e^{2\pi i/n}$. In particular, this forces $f(1)$ to be the number of elements of $X$. 

    The cyclic sieving phenomenon has been observed in many settings. Defant and Lin \cite[Theorem 5.13]{DL24} showed that the order of rowmotion $\Row$ on the $n$th $m$-Tamari lattice (i.e., $ \Tam_n(m) = \Tam(\nu)$ with $\nu = (NE^m)^n$) is $(m+1)n$. Moreover, the triple $(\Tam_n(m), \Row, f(q))$ exhibits the cyclic sieving phenomenon where the polynomial $f(q)$ is the $q$-analogue of the Fuss--Catalan number (see \cite[Section 5]{DL24} for details).

\subsection{Homomesy and homometry}\label{sec.homomesy}

One of the central focuses of dynamical algebraic combinatorics is the property of various statistics on the set of elements in each orbit of an action. If $X$ is a finite set, then a \emph{statistic} on $X$ is a map $\mathsf{st}:X \rightarrow \mathbb{Z}_{\geq 0}$. If $g$ is an invertible map on $X$, then the statistic $\mathsf{st}$ is \emph{homomesic} \cite{PR15} if there is a constant $c$ such that
\begin{equation*}
    \frac{1}{|\mathcal{O}|}\sum_{z\in \mathcal{O}} \mathsf{st}(z) = c,
\end{equation*}
for \emph{every} orbit $\mathcal{O}$ under the map $g$. In other words, the average of the statistic $\mathsf{st}$ is the same for every orbit. 

The statistic $\mathsf{st}$ is called \emph{homometric} \cite{EPRS23} if 
\begin{equation*}
    \sum_{z \in \mathcal{O}} \mathsf{st}(z) = \sum_{z^{\prime} \in \mathcal{O}^{\prime}} \mathsf{st}(z^{\prime}),
\end{equation*}
whenever $\mathcal{O}$ and $\mathcal{O}^{\prime}$ are orbits of the map $g$ of the \emph{same} size. Note that homomesy implies homometry, but the converse is false. 

Consider the \emph{down-degree} statistic on a poset $P$, which is the function $\mathsf{ddeg}:P \rightarrow \mathbb{Z}_{\geq 0}$ given by $\mathsf{ddeg}(p) = |\{ z \in P | z \lessdot p\}|$. Defant and Lin \cite[Theorem 5.14]{DL24} proved that the down-degree statistic on the $n$th $m$-Tamari lattice $\Tam_n(m)$ under rowmotion is homomesic with average $\frac{m(n-1)}{m+1}$. They also provided the following conjecture \cite[Conjecture 10.2]{DL24}:
\begin{conjecture}[\cite{DL24}]\label{conjhomometry}
    Let $\nu$ be a lattice path. The down-degree statistic is homometric for rowmotion on $\Tam(\nu)$.
\end{conjecture}

\section{The alt hook-Tamari lattices}\label{sec.hook}

In this section, we study the alt hook-Tamari lattices $\H_{\delta}(a,b)$, that is, the alt $\nu$-Tamari lattices $\Tam_{\delta}(\nu)$ with $\nu = EN^{a-1}E^{b-1}N$, where $a$ and $b$ are positive integers. Notice that all the $\nu$-paths are contained in the region bounded by $\nu$ and $N^{a}E^{b}$, the Young diagram of the partition $(b,1^{a-1})$, known as the hook shape. In Section \ref{sec.latticehook}, we present the lattice structure of $\H_{\delta}(a,b)$. In Section \ref{sec.orbithook}, we study the orbit structure of $\H_{\delta}(a,b)$ under rowmotion. The orbit sums of several statistics are given in Section \ref{sec.stathook}.

\subsection{The lattice structure of $\H_{\delta}(a,b)$}\label{sec.latticehook}

    Let $a$ and $b$ be positive integers and fix the lattice path $\nu = EN^{a-1}E^{b-1}N$. We recall from Section \ref{sec.tam} that the encoding of $\nu$ as a sequence of nonnegative integers and its increment vector $\delta$ are given by
    \begin{equation}
        \nu= (\nu_0,\nu_1,\dots,\nu_a) = \begin{cases}
            (1,\underbrace{0,\dots,0}_{a-2},b-1,0), & \text{if $a>1$,} \\
            (b,0), & \text{if $a=1$,} 
        \end{cases}  \text{ and }  \delta = \delta(k) = \begin{cases}
            (\underbrace{0,\dots,0}_{a-2},k,0), & \text{if $a>1$.}\\
             0, & \text{if $a=1$,} 
        \end{cases}
    \end{equation}
    where $k=0,1,\dots,b-1$. We would like to point out that there are $b$ distinct increment vectors, each defining a different partial order on the set of $\nu$-paths $P(\nu)$. This results in $b$ different lattice structures of the alt hook-Tamari lattice $\H_{\delta(k)}(a,b)$. When $a=1$, there is only one possible increment vector, namely $\delta =0$, and $\H_{\delta}(1,b)$ is exactly the $\nu$-Dyck lattice with $\nu=E^bN$. 

    By Proposition \ref{prop.int}, each alt hook-Tamari lattice $\H_{\delta(k)}(a,b)$ is isomorphic to the interval $[\nu,1^{\nu}]$ in the $\widehat{\nu}(k)$-Tamari lattice $\Tam(\widehat{\nu}(k))$, where $\widehat{\nu}(k)$ is the lattice path
    \begin{equation}\label{eq.nuhat}
        \widehat{\nu}(k) = \begin{cases}
            (b-k,\underbrace{0,\dots,0}_{a-2},k,0) = E^{b-k}N^{a-1}E^kN, & \text{if $a>1$,}\\
            (b,0) = E^bN, & \text{if $a=1$.}
        \end{cases}
    \end{equation}
    Therefore, every element of $\H_{\delta(k)}(a,b)$ can be naturally identified as an element of $\Tam(\widehat{\nu}(k))$. Following the discussion in Section \ref{sec.bracket}, we can explicitly determine the $\widehat{\nu}(k)$-bracket vector for each such element; they are characterized in Lemma \ref{lem.bracket}.
    
    \begin{example}
        We consider the alt hook-Tamari lattice $\H_{\delta(k)}(a,b)$ with $(a,b,k)=(4,7,3)$, that is, the lattice path $\nu = EN^3E^6N = (1,0,0,6,0)$. We choose the increment vector $\delta=  \delta(3)=(0,0,3,0)$, then the lattice path $\widehat{\nu}(3) = E^4N^3E^3N = (4,0,0,3,0)$. In Figure \ref{fig.althookbracket}, the lattice path $\nu$ is drawn in black while $\widehat{\nu}(3)$ is drawn in blue. Given an element $\mu=N^2ENE^2NE^4$ (shown in red in Figure \ref{fig.althookbracket}) of $\H_{\delta(3)}(4,7)$, by Proposition \ref{prop.int}, we can regard $\mu$ as an element of the Tamari lattice $\Tam(\widehat{\nu}(3))$.
        
        The $\widehat{\nu}(3)$-bracket vector of $\widehat{\nu}(3)$ itself is
        \begin{equation*}
            \mathbf{b}(\widehat{\nu}(3))=(0,0,0,0,\underline{0},\underline{1},\underline{2},3,3,3,\underline{3},\underline{4}).
        \end{equation*}
        The $y$-coordinate of the lattice points along $\mu$ is $0,1,2,2,3,3,3,4,4,4,4,4$. Following the procedure in Section \ref{sec.bracket}, the $\widehat{\nu}(3)$-bracket vector of $\mu$ is then given by
        \begin{equation*}
            \mathbf{b}(\mu)=(4,4,4,2,\underline{0},\underline{1},\underline{2},4,3,3,\underline{3},\underline{4}).
        \end{equation*}
    \end{example}
\begin{figure}[hbt!]
    \centering
    \includegraphics[width=0.3\linewidth]{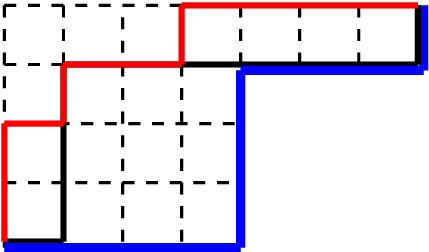}
    \caption{An element $\mu$ (drawn in red) of the alt hook-Tamari lattice $\H_{\delta(k)}(a,b)$ with $(a,b,k)=(4,7,3)$. The element $\mu$ can be identified as an element of the Tamari lattice $\Tam(\widehat{\nu}(3))$, where the lattice path $\widehat{\nu}(3)$ is drawn in blue.}
    \label{fig.althookbracket}
\end{figure}
    \begin{lemma}\label{lem.bracket}
        Let $a,b$ be positive integers and $k$ be an integer with $0 \leq k \leq b-1$. Fix the lattice path $\nu=EN^{a-1}E^{b-1}N$.  Let $\mu$ be an element of $\H_{\delta(k)}(a,b)$. Then the $\widehat{\nu}(k)$-bracket vector of $\mu$ has the following form
        \begin{equation}
            \mathbf{b}(\mu) = \begin{cases} (\alpha,s,\underline{0},\underline{1},\dots,\underline{a-2},\beta,\underline{a-1},\underline{a}), & \text{if $a>1$,}\\
            (\gamma,s,\underline{0},\underline{1}), & \text{if $a=1$},
            \end{cases},
        \end{equation}
        where the non-fixed positions are denoted by $\alpha,\beta,\gamma$, and $s$, and they have the following properties:
        \begin{itemize}
            \item[(1)] if $a>1$, then
            \item $0 \leq s \leq a$, and $\alpha$ contains $b-k-1$ terms while $\beta$ contains $k$ terms, 
            \item the concatenation of $\alpha$ and $\beta$ has the form $(\alpha,\beta)=(\underbrace{a,\dots,a}_{t},\underbrace{a-1,\dots,a-1}_{b-1-t})$,
            where $s$ and $t$ are restricted by
            $\begin{cases}
                0 \leq t \leq b-1-k, &  \text{ if $s=a-1$}, \\
                b-1-k \leq t \leq b-1 , &  \text{ if $s=a$}, \\
                0 \leq t \leq b-1, & \text{ if $0 \leq s \leq a-2$}.
            \end{cases}$
            \item[(2)] if $a=1$, then
            \item $0 \leq s \leq 1$ and $\gamma$ contains $b-1$ terms,
            \item $\gamma$ has the form $\gamma=(\underbrace{1,\dots,1}_{t},\underbrace{0,\dots,0}_{b-1-t})$, where $s$ and $t$ are restricted by $\begin{cases}
                t=b-1 , &  \text{ if $s=1$}, \\
                0 \leq t \leq b-1, & \text{ if $s=0$}.
            \end{cases}$
        \end{itemize}
    \end{lemma}
    \begin{proof}
        We first note that the $\widehat{\nu}(k)$-bracket vector of $\widehat{\nu}(k)$ is given by
        \begin{equation}
            \mathbf{b}(\widehat{\nu}(k)) = \begin{cases}
            (\underbrace{0,\dots,0}_{b-k},\underline{0},\underline{1},\dots,\underline{a-2},\underbrace{a-1,\dots,a-1}_{k},\underline{a-1},\underline{a}), & \text{if $a>1$,}\\
            (\underbrace{0,\dots,0}_{b},\underline{0},\underline{1}), & \text{if $a=1$},
            \end{cases}
        \end{equation}
        and the element $\mu$ of $\H_{\delta(k)}(a,b)$ stays weakly above $\nu$. Now, we record the $y$-coordinates of the lattice points along $\mu$. Based on the procedure mentioned in Section \ref{sec.bracket}, the $\widehat{\nu}(k)$-bracket vector of $\mu$ has the same fixed positions as $\mathbf{b}(\widehat{\nu}(k))$. 

        We first assume $a>1$. The term $s$ in $\mathbf{b}(\mu)$ is determined by the $y$-coordinate of the first $E$ step in $\mu$, ranging from $0$ to $a$ (note that when the $y$-coordinate of the first $E$ step is $a-1$, $s=a-1$ if the number of $E$ steps in $\mu$ that are on $y=a-1$ is more than $k$; otherwise, $s=a$). Thus, $\alpha$ contains $b-k-1$ terms and $\beta$ contains $k$ terms. This shows the first property.

        If $0 \leq s \leq a-2$, $\mu$ must pass through the lattice point $(1,a-1)$, then $(\alpha,\beta)$ consists of $a-1$'s and $a$'s. The number of $a$'s depends on the position of the last $N$ step in $\mu$, which ranges from $0$ to $b-1$. If $s=a-1$, then according to the procedure (2) in Section \ref{sec.bracket}, $\beta$ must consist of $a-1$'s. Thus, $b-1-t \geq k$. Similarly, if $s=a$, then we must have $b-1-t \leq k$. This proves the second property.

        When $a=1$, $s$ is given by the $y$-coordinate of the first $E$ step in $\mu$. It is clear that $s=1$ if and only if $t=b-1$ (this is the case where $\mu = NE^b$). If $s=0$, then $t$ ranges from $0$ to $b-1$.
    \end{proof}

    From Lemma \ref{lem.bracket}, the $\widehat{\nu}(k)$-bracket vector $\mathbf{b}(\mu)$ only depends on two parameters, $s$ and $t$. So, we can denote $\mathbf{b}(\mu)$ by $(s,t)_{\widehat{\nu}(k)}$, called the \textit{simplified $\widehat{\nu}(k)$-bracket vector}, where $s$ is the term right before the fixed position $\underline{0}$, and $t$ is the number of $a$'s in $(\alpha,\beta)$. 

    Given two elements $\mu_1,\mu_2 \in \H_{\delta(k)}(a,b)$, let $(s_1,t_1)_{\widehat{\nu}(k)}$ and $(s_2,t_2)_{\widehat{\nu}(k)}$ be their simplified $\widehat{\nu}(k)$-bracket vectors. It was shown in \cite{CPS20} that $\mu_1 \leq \mu_2$ if and only if $\mathbf{b}(\mu_1) \leq \mathbf{b}(\mu_2)$ as $\widehat{\nu}(k)$-bracket vectors. It is easy to see that $\mathbf{b}(\mu_1) \leq \mathbf{b}(\mu_2)$ if and only if $s_1 \leq s_2$ and $t_1 \leq t_2$. We summarize the above discussion in the following lemma.
    \begin{lemma}\label{lem.partialorder}
        Given two elements $\mu_1,\mu_2 \in \H_{\delta(k)}(a,b)$, let $(s_1,t_1)_{\widehat{\nu}(k)}$ and $(s_2,t_2)_{\widehat{\nu}(k)}$ be their simplified $\widehat{\nu}(k)$-bracket vectors, respectively. Then $\mu_1 \leq \mu_2$ if and only if $s_1 \leq s_2$ and $t_1 \leq t_2$.
    \end{lemma}

    To describe the lattice structure of $\H_{\delta(k)}(a,b)$, we view the simplified $\widehat{\nu}(k)$-bracket vectors $(s,t)_{\widehat{\nu}(k)}$ as the lattice point $(s,t)$ on the plane $\{(x,y)| x,y \geq 0\}$. If $\mu_2$ covers $\mu_1$ in $\H_{\delta(k)}(a,b)$, then we draw an edge connecting their corresponding lattice points $(s_2,t_2)$ with $(s_1,t_1)$. Due to Lemma \ref{lem.partialorder}, the corresponding lattice point of $\mu_2$ lies on the northeastern side of that of $\mu_1$. Let $C_n$ denote the $n$-element chain. The Hasse diagram of $\H_{\delta(k)}(a,b)$ is presented in the following theorem.
    \begin{theorem}\label{thm.Hassehook}
        Let $a,b$ be positive integers. 
        \begin{itemize}
            \item When $k=0$, the Hasse diagram of the alt hook-Tamari lattice $\H_{\delta(k)}(a,b)$ is obtained from the Hasse diagram of the product of two chains $C_a \times C_b$ with one extra vertex $(a,b-1)$ and one extra edge $(a-1,b-1) \lessdot (a,b-1)$; see Figure \ref{fig.hooklattice1}.
            
            \item When $1 \leq k \leq b-1$, the Hasse diagram of the alt hook-Tamari lattice $\H_{\delta(k)}(a,b)$ is obtained from the product of two chains $C_a \times C_b$ with the $k$th square, counted from top to bottom, in the rightmost column deformed into a pentagon; see Figure \ref{fig.hooklattice2}.
        \end{itemize}
    \end{theorem}
\begin{figure}[htbp]
    \centering
    \subfigure[]{\label{fig.hooklattice1}\includegraphics[width=0.3\textwidth]{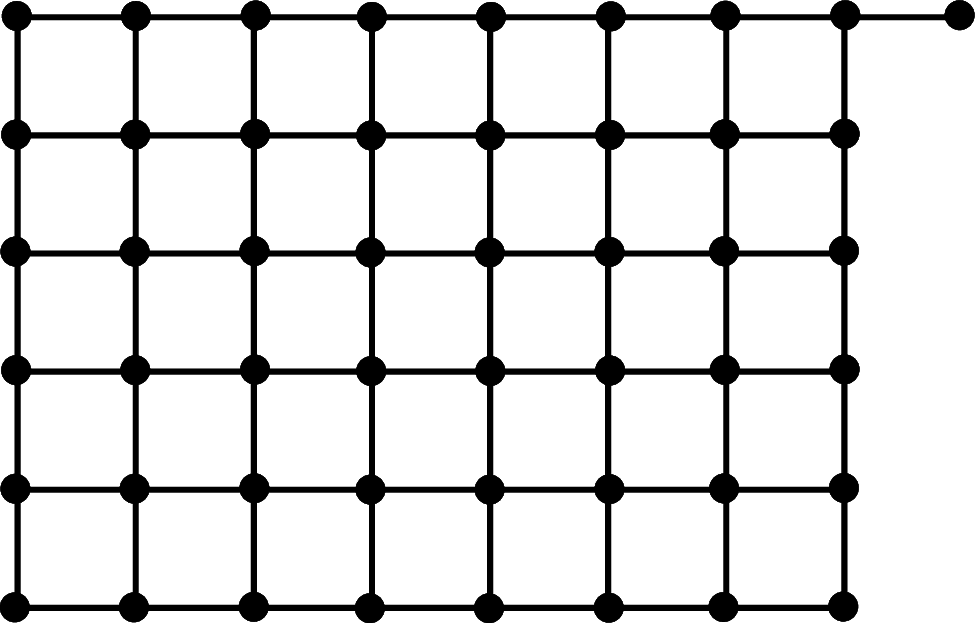}}
    \hspace{15mm}
    \subfigure[]{\label{fig.hooklattice2}\includegraphics[width=0.3\textwidth]{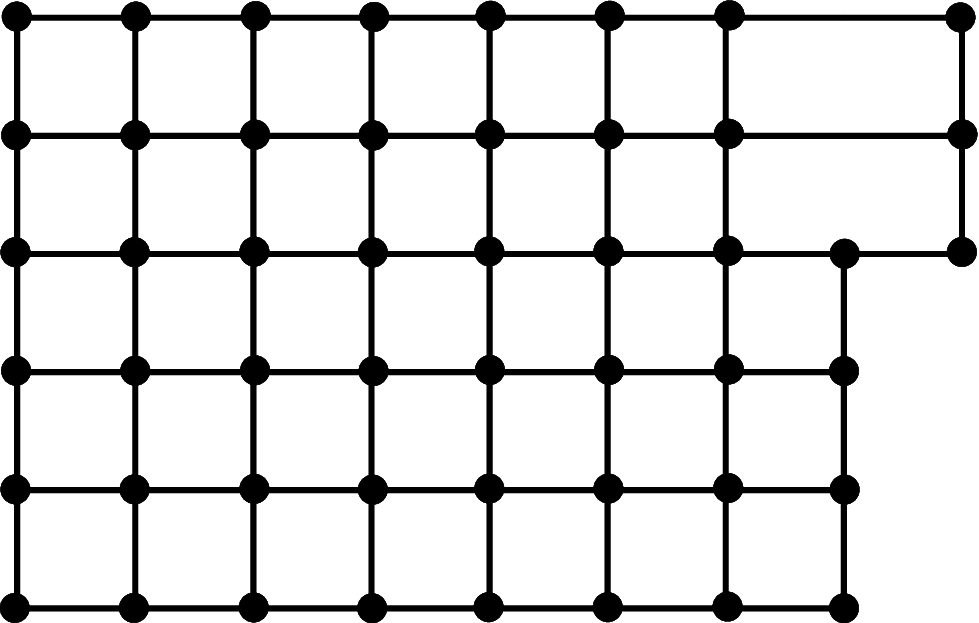}}
    \caption{The Hasse diagram of the alt hook-Tamari lattice $\H_{\delta(k)}(a,b)$ with $(a,b)=(8,6)$ and (a) $k=0$, (b) $k=2$.}
\end{figure}
    
    \begin{proof}
        We first assume $a >1$. It follows from Lemma \ref{lem.bracket} that the simplified $\widehat{\nu}(k)$-bracket vectors correspond to the lattice points $\{(s,t)| 0 \leq s \leq a-2, 0 \leq t \leq b-1\}$, $\{(a-1,t)|0 \leq t \leq b-1-k\}$, and $\{(a,t)|b-1-k \leq t \leq b-1\}$. Lemma \ref{lem.partialorder} tells us how to draw edges between these points. When $k=0$, we have lattice points $\{(s,t)| 0 \leq s \leq a-1,0 \leq t \leq b-1\}$ which form the Hasse diagram of $C_a \times C_b$, with one extra point $(a,b-1)$ connecting with $(a-1,b-1)$. 

        For $1 \leq k \leq b-1$, one can see that the lattice points $(a-1,b-1-k)$ and $(a,b-1-k)$ result in deforming one of the squares in the Hasse diagram of $C_a \times C_b$ into a pentagon; this square is the $k$th square in the rightmost column, counted from top to bottom.

        When $a=1$, it follows again from Lemma \ref{lem.bracket} that the simplified $\widehat{\nu}(k)$-bracket vectors correspond to the lattice points $\{(0,t)| 0 \leq t \leq b-1\} \cup \{(1,b-1)\}$. Lemma \ref{lem.partialorder} implies that the lattice structure of $\H_{\delta}(1,b)$ is the chain of $b+1$ elements, which agrees with our statement (the $k=0$ case).
    \end{proof}

\subsection{The orbit structure of $\H_{\delta}(a,b)$ under rowmotion}\label{sec.orbithook}

    In this section, we prove that the orbit structure of the alt hook-Tamari lattice $\H_{\delta(k)}(a,b)$ under rowmotion is independent of $\delta(k)$. We also explicitly determine their orbit structures and show that they exhibit the cyclic sieving phenomenon. 
    
    By Theorem \ref{thm.Hassehook}, we notice that in the Hasse diagram of $\H_{\delta(k)}(a,b)$, the lattice point $w(k)=(a,b-1-k)$ is join-irreducible and we write $w(k)_{*} = (a-1,b-1-k)$ for the lattice point covered by $w(k)$. One can obtain the Hasse diagram of $C_a \times C_b$ by contracting the edge connecting $w(k)$ and $w(k)_{*}$ from the Hasse diagram of $\H_{\delta(k)}(a,b)$. Since this edge contraction is crucial in our work, we describe it explicitly as follows. 
    
    For convenience, we denote the elements of $C_a \times C_b$ by $\{(x,y) \in \mathbb{Z} \times \mathbb{Z}| 0 \leq x \leq a-1, 0 \leq y \leq b-1\}$. For $0 \leq k \leq b-1$, there are two ways to identify elements under the contraction: in the first one, we delete the join-irreducible element $w(k)$, $\phi_k: \H_{\delta(k)}(a,b) \setminus \{w(k)\} \rightarrow C_a \times C_b$; in the second one, we delete the meet-irreducible element $w(k)_{*}$, $\phi_k^{*}: \H_{\delta(k)}(a,b) \setminus \{w(k)_{*}\} \rightarrow C_a \times C_b$. They are given by
        \begin{equation}\label{eq.contraction}
            \phi_{k}(s,t)=\phi_{k}^{*}(s,t) = \begin{cases}
                (s,t), & \text{ if $0 \leq s \leq a-2$,} \\
                (a-1,t), & \text{ if $s=a-1,a$.} 
            \end{cases}
        \end{equation}
    In particular, $\phi_{k}$ and $\phi^{*}_{k}$ are order-preserving.
    
    Our first main result is stated below. 
    \begin{theorem}\label{thm.rowhook}
        Let $a,b$ be positive integers and $k$ be an integer with $0 \leq k \leq b-1$. The orbit structure of the alt hook-Tamari lattice $\H_{\delta(k)}(a,b)$ under rowmotion is independent of the choices of the increment vector $\delta(k)$. Moreover, there are $g-1$ orbits $\mathcal{O}$ having size $\ell$, and $1$ orbit $\mathcal{O}^{\prime}$ with size $\ell+1$, where $g=\mathrm{gcd}(a,b)$ and $\ell = \mathrm{lcm}(a,b)$.
    \end{theorem}
    \begin{proof}
        We first claim that, in $\H_{\delta(k)}(a,b)$, rowmotion sends the join-irreducible element $w(k)=(a,b-1-k)$ to $w(k)_{*}=(a-1,b-1-k)$. To see this, $\Pop^{\downarrow}(w(k)) = w(k)_{*}$ and by Theorem \ref{thm.charrow}, $\Row(w(k))$ is a maximal element of $\{y \in \H_{\delta(k)}(a,b) | \Pop^{\downarrow}(w(k)) = w(k) \wedge y\}$, which is $w(k)_{*}$. Note that this result is independent of the increment vector $\delta(k)$. 
  
        For each $\delta(k)$, we apply the contraction $\phi_k$ to $\H_{\delta(k)}(a,b)$ and obtain $C_a \times C_b$. Since these two Hasse diagrams only differ in a single edge connecting $w(k)$ and $w(k)_{*}$, it is clear that for $z \neq w(k)$, $\phi_k \left( \Row_{\H_{\delta(k)}(a,b)}(z) \right) = \Row_{C_a \times C_b} (\phi_k(z))$. Thus, the collection of orbits $\{\mathcal{O}_i\}_{i=1}^{m}$ on $\H_{\delta(k)}(a,b)$ under rowmotion is almost the same as the collection of orbits $\{\mathcal{O}^{\prime}_i\}_{i=1}^{m}$ of $C_a \times C_b$ under rowmotion, with the exception of a single orbit. Let $\mathcal{O}^{\prime}_m$ denote the orbit containing $w(k)_{*}$, then $\mathcal{O}_m = \mathcal{O}^{\prime}_m \cup \{w(k)\}$, while $\mathcal{O}_i = \mathcal{O}^{\prime}_i$ for $i<m$.
    
        Next, we determine the orbit structure of $C_a \times C_b$ under rowmotion. Following Theorem \ref{thm.charrow}, rowmotion on $C_a \times C_b$ is given by
        \begin{equation}\label{eq.rowhookmod}
            \Row(x,y) = (x^{\prime},y^{\prime}), \text{ where $x^{\prime} \equiv x-1 \pmod{a}$ and $y^{\prime} \equiv y-1 \pmod{b}$}.
        \end{equation}
        To see this, if $x,y >0$, then $\Pop^{\downarrow}(x,y) = (x-1,y) \wedge (x,y-1) = (x-1,y-1)$, and $\Row(x,y)$ is given by a maximal element of $\{(x^{\prime},y^{\prime}) \in C_a \times C_b| (x-1,y-1) = (x,y) \wedge (x^{\prime},y^{\prime})\}$. Thus, $\Row(x,y) = (x-1,y-1)$. If $x=0,y>0$, then $\Pop^{\downarrow}(0,y) = (0,y-1)$ and $\Row(0,y) = \max\{(x^{\prime},y^{\prime}) \in C_a \times C_b| (0,y-1) = (0,y) \wedge (x^{\prime},y^{\prime})\}$. Thus, $\Row(0,y) = (a-1,y-1)$. We omit the case $y=0,x>0$, since this can be argued similarly. Finally, it is easy to see that $\Row(0,0) = (a-1,b-1)$. Therefore, we obtain \eqref{eq.rowhookmod}.
        
        Determining the orbit structure of rowmotion on $C_a \times C_b$ is equivalent to finding the minimal positive integer $d$ such that $x-d \equiv x \pmod{a}$ and $y-d \equiv y \pmod{b}$. This implies that  $d = \mathrm{lcm}(a,b)$ is the size of each orbit, while the number of orbits is given by $\mathrm{gcd}(a,b)$. To recover the orbit structure of rowmotion on $\H_{\delta(k)}(a,b)$, we simply put the element $w(k)$ involved in the edge contraction back in the corresponding orbit. Therefore, the orbit structure of rowmotion on $\H_{\delta(k)}(a,b)$ is independent of $\delta(k)$, and we have $\mathrm{gcd}(a,b)-1$ orbits having size $\mathrm{lcm}(a,b)$ and one orbit with size $\mathrm{lcm}(a,b)+1$, as desired.
        \end{proof}
    
    As a corollary of Theorem \ref{thm.rowhook}, we find a polynomial $f_{a,b}(q)$ such that the rowmotion on alt hook-Tamari lattices exhibits the cyclic sieving phenomenon (Section \ref{sec.csp}). For convenience, we denote the order of rowmotion by $\mathsf{or}(\Row)$.
    \begin{corollary}\label{cor.csp}
        With the same notations as in Theorem \ref{thm.rowhook}, rowmotion $\mathrm{Row}: \H_{\delta(k)}(a,b) \rightarrow \H_{\delta(k)}(a,b)$ has order
        \begin{equation}\label{eq.hookorder}
            \mathsf{or}(\Row) = \begin{cases}
            \ell(\ell+1), & \text{if $g>1$,}\\
            \ell+1, & \text{if $g=1$.}
        \end{cases}
        \end{equation}
        Moreover, the triple $(\H_{\delta(k)}(a,b), \mathrm{Row}, f_{a,b}(q))$ exhibits the cyclic sieving phenomenon, where 
        \begin{equation}
            f_{a,b}(q) = 
                \sum_{j=0}^{\ell} q^{j \ell} + (g-1)\sum_{j=0}^{\ell-1} q^{j(\ell+1)}.
        \end{equation}
    \end{corollary}
    \begin{proof}
        The order of rowmotion \eqref{eq.hookorder} follows immediately from Theorem \ref{thm.rowhook}. On one hand, the number of elements that are fixed by applying rowmotion $d$ times is given by
        \begin{equation}\label{eq.csp1}
            |\{\mu \in \H_{\delta(k)}(a,b) | \Row^d(\mu) = \mu\}| = \begin{cases}
                g\ell+1, & \text{ if $d \equiv 0 \pmod{ \mathsf{or}(\Row)}$}, \\
                \ell+1, & \text{ if $\ell+1 \mid d$}, \\
                (g-1)\ell, & \text{ if $\ell \mid d$}, \\
                0, & \text{ otherwise}. 
            \end{cases}
        \end{equation}
        
        On the other hand, we take $\omega = \begin{cases}
            e^{2 \pi i / \ell(\ell+1)}, &\text{ if $g>1$,}\\
            e^{2 \pi i / (\ell+1)}, & \text{ if $g=1$}.
        \end{cases}$ 
        
        When $g>1$, we have
        \begin{equation}\label{eq.csp2}
            \sum_{j=0}^{\ell} (\omega^d)^{j \ell} = \sum_{j=0}^{\ell} e^{2 \pi i \cdot jd/(\ell+1)} = \begin{cases}
                0, & \text{ if $\ell+1 \nmid d$},\\
                \ell+1, & \text{ if $\ell+1 \mid d$}.
            \end{cases}
        \end{equation}
        \begin{equation}\label{eq.csp3}
            \sum_{j=0}^{\ell-1} (\omega^d)^{j (\ell+1)} = \sum_{j=0}^{\ell-1} e^{2 \pi i \cdot jd/\ell} = \begin{cases}
                0, & \text{ if $\ell \nmid d$},\\
                \ell, & \text{ if $\ell \mid d$}.
            \end{cases}
        \end{equation}

        When $g=1$, we have 
        \begin{equation}\label{eq.csp2-1}
            \sum_{j=0}^{\ell} (\omega^d)^{j \ell} = \sum_{j=0}^{\ell} e^{2 \pi i \cdot jd\ell/(\ell+1)} = \sum_{j=0}^{\ell} e^{-2 \pi i \cdot jd/(\ell+1)} = \begin{cases}
                0, & \text{ if $\ell+1 \nmid d$},\\
                \ell+1, & \text{ if $\ell+1 \mid d$}.
            \end{cases}
        \end{equation}
        
        Using \eqref{eq.csp2}, \eqref{eq.csp3} and \eqref{eq.csp2-1}, and comparing $f_{a,b}(\omega^d)$ with \eqref{eq.csp1}, we obtain 
        \begin{equation}
            f_{a,b}(\omega^d) = |\{\mu \in \H_{\delta(k)}(a,b) | \Row^d(\mu) = \mu\}|,
        \end{equation}
        for $d=0,1,\dots,\mathsf{or}(\Row)-1$. Therefore, the triple $(\H_{\delta(k)}(a,b), \mathrm{Row}, f_{a,b}(q))$ exhibits the cyclic sieving phenomenon.
    \end{proof}

\subsection{The statistics on $\H_{\delta}(a,b)$}\label{sec.stathook}

    Recall that the down-degree statistic of an element $p$ of a poset $P$ is the function $\mathsf{ddeg}:P \rightarrow \mathbb{Z}_{\geq 0}$ given by $\mathsf{ddeg}(p) = |\{ z \in P | z \lessdot p\}|$ (Section \ref{sec.homomesy}). We also consider several standard statistics on lattice paths. For a $\nu$-path $\mu$, let $\mathsf{area}(\mu)$ denote the area of the region bounded by $\mu$ and $\nu$; let $\mathsf{peak}(\mu)$ be the number of peaks of $\mu$ (i.e., occurrences of the pattern $NE$ in $\mu$); let $\mathsf{val}(\mu)$ be the number of valleys of $\mu$ (i.e., occurrences of the pattern $EN$ in $\mu$). 
    
    Theorem \ref{thm.rowhook} shows that the orbit structure of rowmotion on $\H_{\delta(k)}(a,b)$ does not depend on the choice of the increment vector $\delta(k)$. While some elements contained in each orbit may vary with $\delta(k)$, we prove in the following theorem that the orbit sums of the statistics $\mathsf{ddeg}, \mathsf{peak}$, and $\mathsf{val}$ remain independent of $\delta(k)$. In fact, these three statistics are homometric under rowmotion. For intermediate cases in general, the statistic $\mathsf{area}$ is not homometric; see Figure \ref{fig.alttam} (when $(a,b,k) = (3,3,1)$) for example.
    \begin{theorem}\label{thm.stathook}
        With the same notations as in Theorem \ref{thm.rowhook}, rowmotion on $\H_{\delta(k)}(a,b)$ has the following properties:
        \begin{itemize}
            \item[(1)] The down-degree statistic is independent of the increment vector $\delta(k)$,
                \begin{equation}\label{eq.hookddegformula}
                    \sum_{z \in \mathcal{O}} \mathsf{ddeg}(z) = \frac{2ab-a-b}{g}, \text{ and } \sum_{z \in \mathcal{O}^{\prime}} \mathsf{ddeg}(z) = \frac{2ab-a-b}{g}+1.
                \end{equation}
                Moreover, $\mathsf{ddeg}$ is homometric under rowmotion.
            \item[(2)] The peak statistic is independent of the increment vector $\delta(k)$,
                \begin{equation}
                    \sum_{z \in \mathcal{O}} \mathsf{peak}(z) = \frac{3ab-2a-2b}{g}, \text{ and } \sum_{z \in \mathcal{O}^{\prime}} \mathsf{peak}(z) = \frac{3ab-2a-2b}{g}+2.
                \end{equation}
                Moreover, $\mathsf{peak}$ is homometric under rowmotion.
            \item[(3)] The valley statistic is independent of the increment vector $\delta(k)$,
                \begin{equation}
                    \sum_{z \in \mathcal{O}} \mathsf{val}(z) = \frac{2ab-a-b}{g}, \text{ and } \sum_{z \in \mathcal{O}^{\prime}} \mathsf{val}(z) = \frac{2ab-a-b}{g}+1.
                \end{equation}
            Moreover, $\mathsf{val}$ is homometric under rowmotion.
            \item[(4)] The area statistic is homometric in two cases: when $k=b-1$ (i.e., the hook-Tamari lattice), 
                \begin{equation}\label{eq.areab-1main}
                    \sum_{z \in \mathcal{O}} \mathsf{area}(z) = \ell \cdot \left(\frac{a+b-2}{2} + \frac{1}{a} \right), \text{ and } \sum_{z \in \mathcal{O}^{\prime}} \mathsf{area}(z) = \ell \cdot \left(\frac{a+b-2}{2} + \frac{1}{a} \right) + (a-1),
                \end{equation}
            and when $k=0$ (i.e., the $\nu$-Dyck lattice with $\nu=EN^{a-1}E^{b-1}N$), 
                \begin{equation}\label{eq.area0main}
                    \sum_{z \in \mathcal{O}} \mathsf{area}(z) = \ell \cdot \left(\frac{a+b-2}{2} \right), \text{ and } \sum_{z \in \mathcal{O}^{\prime}} \mathsf{area}(z) = \ell \cdot \left(\frac{a+b-2}{2} \right) + (a+b-1).
                \end{equation}
        \end{itemize}
    \end{theorem}
    \begin{proof}
        Following the idea in the proof of Theorem \ref{thm.rowhook}, we can view the statistics of elements on $\H_{\delta(k)}(a,b)$ as the statistics of the corresponding lattice points on $C_a \times C_b$ via the map $\phi_k$ or $\phi_k^{*}$ \eqref{eq.contraction}. It suffices to compute the total value of the statistics of elements in each orbit of $C_a \times C_b$ under rowmotion, and then find the statistics of the element $w(k)$ or $w(k)_{*}$ involved in the edge contraction. Given a statement $P$, we write $\varepsilon(P)$ for the indicator function of $P$: $\varepsilon(P) = 1$ if $P$ is true, while $\varepsilon(P)=0$ if $P$ is false. 
        \begin{itemize}
            \item[(1)] We apply the map $\phi_k$ to $\H_{\delta(k)}(a,b)$. It is easy to see that the statistic $\mathsf{ddeg}$ on $C_a \times C_b$ is given by
            \begin{equation}\label{eq.rowddeg}
                \mathsf{ddeg}(x,y) = 2-\varepsilon(x=0)-\varepsilon(y=0),
            \end{equation} 
            for $0 \leq x \leq a-1$ and $0 \leq y \leq b-1$. Thus, summing up the down-degree statistic of all elements in the orbit $\mathcal{O}$ (of size $\ell$), we obtain
            \begin{align}
                \sum_{z \in \mathcal{O}} \mathsf{ddeg}(z) & = \sum_{i=1}^{\mathrm{lcm}(a,b)} \mathsf{ddeg}\left( \mathrm{Row}^{i}(x,y) \right) \nonumber \\
                 & = \sum_{i=1}^{\mathrm{lcm}(a,b)} ( 2-\varepsilon(a|x-i) - \varepsilon(b|y-i) ) \nonumber \\
                 & = 2 \cdot \mathrm{lcm}(a,b) - \frac{b}{\mathrm{gcd}(a,b)} - \frac{a}{\mathrm{gcd}(a,b)} \nonumber \\
                 & = \frac{2ab-a-b}{\mathrm{gcd}(a,b)}, \nonumber
            \end{align}
            where the second equality follows from \eqref{eq.rowhookmod} and \eqref{eq.rowddeg}. Since $w(k)$ is join-irreducible, $\mathsf{ddeg}(w(k)) =1$. Putting the element $w(k)$ in the orbit $\mathcal{O}^{\prime}$ increases the total down-degree by 1. This gives the desired down-degree formula.
            \item[(2)] With the simplified $\widehat{\nu}(k)$-bracket vector $\mathbf{b}(\mu) = (s,t)_{\widehat{\nu}(k)}$, the statistic $\mathsf{peak}$ on $\nu$-paths $\mu$ is given by
            \begin{equation}\label{eq.peakhook}
                \mathsf{peak}(s,t) = 3- \varepsilon(s=0) - \varepsilon(s=a-1) - \varepsilon(s=a) - \varepsilon(t=0) - \varepsilon(t=b-1),
            \end{equation}
            where the restriction of $s$ and $t$ are given in Lemma \ref{lem.bracket}. To see this, we observe that the maximum number of peaks of $\mu$ is $3$. If $\mathbf{b}(\mu) = (a,b-1)_{\widehat{\nu}(k)}$, then $\mu$ is the top path $1^{\nu}$, which contains exactly one peak. For the remaining cases, the number of peaks is determined by the boundary conditions of $\mathbf{b}(\mu)$:
            \begin{itemize}
                \item If the first component of $\mathbf{b}(\mu)$ is $0$ (resp., $a-1$ or $a$), then the first $E$ step of $\mu$ lies on $y=0$ (resp., $y=a-1$); this eliminates one potential peak.
                \item Similarly, if the second component of $\mathbf{b}(\mu)$ is $0$ (resp., $b-1$), then the last $N$ step of $\mu$ lies on $x=b$ (resp., $x=1$), which again reduces the number of peaks by one. 
            \end{itemize}
            This gives \eqref{eq.peakhook}.

            For $1 \leq k \leq b-1$, we apply the map $\phi_{k}$ to obtain $C_a \times C_b$; however, for the case $k=0$, we apply $\phi^{*}_{0}$ to yield the same structure. This is because $\mathsf{peak}(w(k)) = 2-\varepsilon(k=0)$ and $\mathsf{peak}(w(0)_{*}) = 2$, we can guarantee that putting the element back in the orbit $\mathcal{O}^{\prime}$ increases the total number of peaks by $2$. 
            
            After applying $\phi_k$ and $\phi^{*}_0$, we may write the statistic $\mathsf{peak}$ of the elements of $C_a \times C_b$ by
            \begin{equation}\label{eq.peak}
                \mathsf{peak}(x,y) = 3- \varepsilon(x=0) - \varepsilon(x=a-1)- \varepsilon(y=0) - \varepsilon(y=b-1),
            \end{equation}
            for $0 \leq x \leq a-1$ and $0 \leq y \leq b-1$. Thus, summing up the number of peaks of all elements in the orbit $\mathcal{O}$ (of size $\ell$), we obtain
            \begin{align*}
                \sum_{z \in \mathcal{O}}\mathsf{peak}(z) & = \sum_{i=1}^{\mathrm{lcm}(a,b)} \mathsf{peak}(\Row^{i}(x,y)) \\
                & = \sum_{i=1}^{\mathrm{lcm}(a,b)} (3- \varepsilon(a|x-i) - \varepsilon(a|x+1-i)- \varepsilon(b|y-i) - \varepsilon(b|y+1-i))\\
                & =3 \cdot \mathrm{lcm}(a,b) - \frac{2b}{\mathrm{gcd}(a,b)} - \frac{2a}{\mathrm{gcd}(a,b)} \\
                & = \frac{3ab-2a-2b}{\mathrm{gcd}(a,b)},
            \end{align*}
            where the second equality follows from \eqref{eq.rowhookmod} and \eqref{eq.peak}. Combining the above discussion, we obtain the desired formula for the statistic $\mathsf{peak}$.

            \item[(3)] Let $\mathbf{b}(\mu)=(s,t)_{\widehat{\nu}(k)}$ be the simplified $\widehat{\nu}(k)$-bracket vector of a $\nu$-path $\mu$. The statistic $\mathsf{val}$ of $\mu$ is given by
            \begin{equation}
                \mathsf{val}(s,t)=2-\varepsilon(s=a)-\varepsilon(s=a-1)-\varepsilon(t=b-1).
            \end{equation}
            We omit the details here, as the analysis follows an argument similar to the one presented for the number of peaks.

            Again, for $1 \leq k \leq b-1$, we apply the map $\phi_{k}$ to obtain $C_a \times C_b$; however, for the case $k=0$, we apply $\phi^{*}_{0}$ to yield the same structure. This is because $\mathsf{val}(w(k)) = 1-\varepsilon(k=0)$ and $\mathsf{val}(w(0)_{*}) = 1$, we can guarantee that putting the element back in the orbit $\mathcal{O}^{\prime}$ increases the total number of valleys by $1$.  
            
            After applying $\phi_k$ and $\phi^{*}_0$, we may write the statistic $\mathsf{val}$ of the elements of $C_a \times C_b$ by
            \begin{equation}\label{eq.val}
                \mathsf{val}(x,y) = 2 - \varepsilon(x=a-1)- \varepsilon(y=b-1),
            \end{equation}
            for $0 \leq x \leq a-1$ and $0 \leq y \leq b-1$. Thus, summing up the number of valleys of all elements in the orbit $\mathcal{O}$ (of size $\ell$), we obtain
            \begin{align*}
                \sum_{z \in \mathcal{O}}\mathsf{val}(z) & = \sum_{i=1}^{\mathrm{lcm}(a,b)} \mathsf{val}(\Row^{i}(x,y)) \\
                & = \sum_{i=1}^{\mathrm{lcm}(a,b)} (2 - \varepsilon(a|x+1-i)- \varepsilon(b|y+1-i))\\
                & =2 \cdot \mathrm{lcm}(a,b) - \frac{b}{\mathrm{gcd}(a,b)} - \frac{a}{\mathrm{gcd}(a,b)} \\
                & = \frac{2ab-a-b}{\mathrm{gcd}(a,b)},
            \end{align*}
            where the second equality follows from \eqref{eq.rowhookmod} and \eqref{eq.val}. Combining the above discussion, we obtain the desired formula for the statistic $\mathsf{val}$.  

            \item[(4)] Let $\mu$ be a $\nu$-path with the simplified $\widehat{\nu}(k)$-bracket vector $\mathbf{b}(\mu)=(s,t)_{\widehat{\nu}(k)}$. We claim that 
            \begin{equation}
                \mathsf{area}(s,t) = s+t.
            \end{equation}
            If $\mathbf{b}(\mu)=(a,b-1)_{\widehat{\nu}(k)}$, then $\mu$ is the top path $1^{\nu}$ and the area is $a+b-1$. For the remaining cases, the area of $\mu$ depends on the $y$-coordinate of the first $E$ step and the $x$-coordinate of the last $N$ step of $\mu$, this information is given by its simplified $\widehat{\nu}(k)$-bracket vector $(s,t)_{\widehat{\nu}(k)}$. A straightforward check shows that $\mathsf{area}(s,t) = s+t$.

            After applying $\phi_{k}$, the area statistic of the elements of $C_a \times C_b$ can be written as
            \begin{equation}\label{eq.areat1}
                \mathsf{area}(x,y) = x+y+\varepsilon(x=a-1 \wedge y \geq b-k),
            \end{equation}
            for $0 \leq x \leq a-1$ and $0 \leq y \leq b-1$. Since the area depends on $k$, the total area statistic of the elements in an orbit depends on $\delta(k)$, and it is not homometric in general. However, we show below that when $k=0$ and $k=b-1$, the area statistic is homometric. 
            
            When $k=0$, \eqref{eq.areat1} simply reduces to $\mathsf{area}(x,y) = x+y$. Also, $\mathsf{area}(w(0)) = a+b-1$. Summing up the area of all elements in the orbit $\mathcal{O}$ (of size $\ell$), we obtain
            \begin{align*}
                \sum_{z \in \mathcal{O}} \mathsf{area}(z) & = \sum_{i=1}^{\mathrm{lcm}(a,b)} \mathsf{area}(\Row^{i}(x,y)) \\
                & = \sum_{i=1}^{\mathrm{lcm}(a,b)} \left( (x-i) \pmod{a} + (y-i) \pmod{b} \right) \\
                & = (0+1+\cdots+a-1)\cdot\frac{\mathrm{lcm}(a,b)}{a} + (0+1+\cdots+b-1)\cdot\frac{\mathrm{lcm}(a,b)}{b} \\
                & = \mathrm{lcm}(a,b) \left(\frac{a+b-2}{2} \right),
            \end{align*}
            where the notation $(x-i) \pmod{a}$ indicates the remainder of dividing $x-i$ by $a$. This proves \eqref{eq.area0main}.
            
            When $k=b-1$, we instead apply the map $\phi^{*}_{b-1}$. This results in $\mathsf{area}(x,y) = x+y+\varepsilon(x=a-1)$. Also, $\mathsf{area}(w(b-1)_{*}) = a-1$. Summing up the area of all elements in the orbit $\mathcal{O}$ (of size $\ell$), we obtain 
            \begin{align*}
                \sum_{z \in \mathcal{O}} \mathsf{area}(z) & = \sum_{i=1}^{\mathrm{lcm}(a,b)} \mathsf{area}(\Row^{i}(x,y)) \\
                & = \sum_{i=1}^{\mathrm{lcm}(a,b)} \left( (x-i) \pmod{a} + (y-i) \pmod{b} + \varepsilon(a|x+1-i) \right) \\
                & = (0+1+\cdots+a-1)\cdot\frac{\mathrm{lcm}(a,b)}{a} + (0+1+\cdots+b-1)\cdot\frac{\mathrm{lcm}(a,b)}{b} + \frac{b}{\mathrm{gcd}(a,b)} \\
                & = \mathrm{lcm}(a,b) \left(\frac{a+b-2}{2} + \frac{1}{a}\right).
            \end{align*}
            This proves \eqref{eq.areab-1main}. 
        \end{itemize}
        We complete the proof of Theorem \ref{thm.stathook}.
    \end{proof}
\begin{remark}\label{rmk}
    We notice that the rowmotion orbit structures of the alt hook-Tamari lattices coincide with those of the antichains of the fence posets with two segments studied by Elizalde, Plante, Roby, and Sagan \cite{EPRS23}. Specifically, our down-degree formula \eqref{eq.hookddegformula} coincides with the orbit sum of their statistic $\chi$, which counts the total number of antichain elements in an orbit. Our area formula \eqref{eq.area0main} also coincides with the orbit sum of their statistic $\hat{\chi}$, which counts the total number of order ideal elements in an orbit. Their results were presented in \cite[Theorem 4.3]{EPRS23}. 
    
    In fact, there is a simple bijection between the elements of the hook-Dyck lattice (i.e., $\mathsf{Dyck}(\nu)$ with $\nu=EN^{a-1}E^{b-1}N$) and the antichains of the fence poset with two segments of lengths $a-1$ and $b-1$. The bijection maps the element $(x,y)$ of the hook-Dyck lattice to the antichain containing the $x$th element on the left segment and the $y$th element on the right segment (counted from the bottom) of the fence poset. It also maps the maximal element of the hook-Dyck lattice to the antichain containing the maximal element of the fence poset. It is straightforward to check that the down-degree (resp., area) statistic on the hook-Dyck lattices coincides with the statistic $\chi$ (resp., $\hat{\chi}$) on the fence poset. 
\end{remark}

\section{Alt $2$-row-Tamari lattices $\T_{\delta}(a,b)$}\label{sec.2row}

In this section, we study the alt $2$-row-Tamari lattices $\T_{\delta}(a,b)$, that is, the alt $\nu$-Tamari lattices $\Tam_{\delta}(\nu)$ with $\nu = E^{a}NE^{b}N$, where $a$ and $b$ are nonnegative integers. In this case, all the $\nu$-paths are contained in the region bounded by $\nu$ and $N^2E^{a+b}$, the Young diagram of the partition $(a+b,a)$. 

\subsection{The lattice structure of $\T_{\delta}(a,b)$}

Let $a$ and $b$ be nonnegative integers and fix the lattice path $\nu=E^aNE^bN$. The encoding of $\nu$ as a sequence of nonnegative integers and its increment vector (Section \ref{sec.tam}) are given by
\begin{equation}
    \nu=(\nu_0,\nu_1,\nu_2)=(a,b,0) \quad \text{and} \quad \delta=\delta(k)=(k,0),
\end{equation}
where $k=0,1,\dots,b$. There are $b+1$ choices of the increment vectors, each defining a partial order on the set of $\nu$-paths $P(\nu)$. This gives $b+1$ different lattice structures of the alt $2$-row-Tamari lattice $\T_{\delta(k)}(a,b)$. 

Again, by Proposition \ref{prop.int}, each alt $2$-row-Tamari lattice is isomorphic to the interval $[\nu,1^{\nu}]$ in the $\widehat{\nu}(k)$-Tamari lattice $\Tam(\widehat{\nu}(k))$, where $\widehat{\nu}(k)$ is a lattice path given by
\begin{equation}
    \widehat{\nu}(k)=(a+b-k,k,0) = E^{a+b-k}NE^{k}N.
\end{equation}
Therefore, every element of $\T_{\delta(k)}(a,b)$ can be naturally identified as an element of $\Tam(\widehat{\nu}(k))$. We characterize their $\widehat{\nu}(k)$-bracket vectors in Lemma \ref{lem.2rowbracket}.

\begin{example}
    We consider the alt $2$-row-Tamari lattice $\T_{\delta(k)}(a,b)$ with $(a,b,k)=(3,3,2)$, that is, the lattice path $\nu=E^3NE^3N = (3,3,0)$. We select the increment vector $\delta=\delta(2)=(2,0)$, then the lattice path $\widehat{\nu}(2) =  E^4NE^2N$. In Figure \ref{fig.alt2rowbracket}, the lattice paths $\nu$ and $\widehat{\nu}(2)$ are drawn in black and blue, respectively. Given an element $\mu = NE^2NE^4$ (shown in red in Figure \ref{fig.alt2rowbracket}) of $\T_{\delta(2)}(3,3)$, we can view $\mu$ as an element of the Tamari lattice $\Tam(\widehat{\nu}(2))$ (by Proposition \ref{prop.int}). 

    The $\widehat{\nu}(2)$-bracket vector of $\widehat{\nu}(2)$ itself is
    \begin{equation*}
        \mathbf{b}(\widehat{\nu}(2)) = (0,0,0,0,\underline{0},1,1,\underline{1},\underline{2}).
    \end{equation*}
    The $y$-coordinate of the lattice points along $\mu$ is $0,1,1,1,2,2,2,2,2$. Following the procedure in Section \ref{sec.bracket}, the $\widehat{\nu}(2)$-bracket vector of $\mu$ is given by
    \begin{equation*}
        \mathbf{b}(\mu)=(2,2,2,2,\underline{0},1,1,\underline{1},\underline{2}).
    \end{equation*}
\end{example}

\begin{figure}[hbt!]
    \centering
    \includegraphics[width=0.3\linewidth]{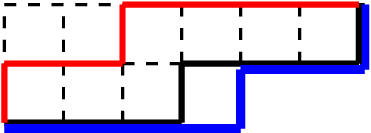}
    \caption{An element $\mu$ (drawn in red) of the alt $2$-row-Tamari lattice $\T_{\delta(k)}(a,b)$ with $(a,b,k)=(3,3,2)$. The element $\mu$ can be identified as an element of the Tamari lattice $\Tam(\widehat{\nu}(2))$, where the lattice path $\widehat{\nu}(2)$ is drawn in blue.}
    \label{fig.alt2rowbracket}
\end{figure}

\begin{lemma}\label{lem.2rowbracket}
    Let $a,b$ and $k$ be nonnegative integers with $0 \leq k \leq b$. Fix the lattice path $\nu = E^aNE^bN$. Let $\mu$ be an element of $\T_{\delta(k)}(a,b)$. Then the $\widehat{\nu}(k)$-bracket vector of $\mu$ has the following form
    \begin{equation}\label{eq.2rowbracket}
        \mathbf{b}(\mu)=(\alpha,\underbrace{0,\dots,0}_{s},\underline{0},\beta,\underline{1},\underline{2}).
    \end{equation}
    Here $\alpha$ and $\beta$ denote the blocks of  non-fixed nonzero entries, and $s$ denotes the number of non-fixed zero entries. They satisfy
    \begin{itemize}
        \item[(1)] $0\leq s \leq a$, and $\alpha$ contains $a+b-k-s$ terms while $\beta$ contains $k$ terms,
        \item[(2)] the concatenation of $\alpha$ and $\beta$ has the form
        \begin{equation*}
            (\alpha,\beta)=(\underbrace{2,\dots,2}_{t},\underbrace{1,\dots,1}_{a+b-s-t}),
        \end{equation*}
        where $0 \leq t \leq a+b-s$.
    \end{itemize}
\end{lemma}
\begin{proof}
    We first note that the $\widehat{\nu}(k)$-bracket vector of $\widehat{\nu}(k)$ is given by
    \begin{equation}
        \mathbf{b}(\widehat{\nu}(k)) = (\underbrace{0,\dots,0}_{a+b-k},\underline{0},\underbrace{1,\dots,1}_{k},\underline{1},\underline{2}),
    \end{equation}
    and the element $\mu$ of $\T_{\delta(k)}(a,b)$ stays weakly above $\nu$. Now, we record the $y$-coordinates of the lattice points along $\mu$, based on the procedure mentioned in Section \ref{sec.bracket}, the $\widehat{\nu}(k)$-bracket vector of $\mu$ has the same fixed positions as $\mathbf{b}(\widehat{\nu}(k))$.

    In \eqref{eq.2rowbracket}, the number of non-fixed $0$'s $s$ is determined by the number of first consecutive $E$ steps starting from the origin, so $0 \leq s \leq a$. Thus, $\alpha$ contains $a+b-k-s$ terms and $\beta$ contains $k$ terms. This shows the first property.

    The entries of $\alpha$ and $\beta$ are determined by the number of $E$ steps of $\mu$ lying on $y=1$ and $y=2$. According to the procedure $(2)$ in Section \ref{sec.bracket}, the concatenation of $\alpha$ and $\beta$ consists of $t$ copies of $2$'s and then followed by $a+b-s-t$ copies of $1$'s, where $0 \leq t \leq a+b-s$. This shows the second property.
\end{proof}

For simplicity, we record the $\widehat{\nu}(k)$-bracket vector $\mathbf{b}(\mu)$ from Lemma \ref{lem.2rowbracket} by a triple $(a-s,u,v)$, where $u$ is the number of $2$'s in $\alpha$ and $v$ is the number of $2$'s in $\beta$. Note that $\beta$ contains at least one $2$ (i.e., $v>0$) if and only if $\alpha$ contains only $2$'s (i.e., $u=a+b-k-s$). In other words, when $s$ is given, all the possible pairs $(u,v)$ in the triple are given by
\begin{equation}\label{eq.2rowpartialorder}
    \{(i,0)|i=0,1,\dots,a+b-k-s\} \cup \{(a+b-k-s,j)|j=1,\dots,k\}.
\end{equation}

\begin{lemma}\label{lem.partialorder2row}
    Given two elements $\mu_1,\mu_2 \in \T_{\delta(k)}(a,b)$, let $(a-s_1,u_1,v_1)$ and $(a-s_2,u_2,v_2)$ be their simplified triples, respectively. Then $\mu_1 \leq \mu_2$ if and only if $s_1 \geq s_2$, $u_1 \leq u_2$, and $v_1 \leq v_2$. 
\end{lemma}
\begin{proof}
    It suffices to show that $\mathbf{b}(\mu_1) \leq \mathbf{b}(\mu_2)$ as $\widehat{\nu}(k)$-bracket vectors if and only if $s_1 \geq s_2$, $u_1 \leq u_2$, and $v_1 \leq v_2$. By Lemma \ref{lem.2rowbracket}, we may write 
    \begin{equation*}
        \mathbf{b}(\mu_i)=(\alpha_i,\underbrace{0,\dots,0}_{s_i},\underline{0},\beta_i,\underline{1},\underline{2}), \quad \text{for $i=1,2$.}
    \end{equation*} 
    
    Note that $\mathbf{b}(\mu_1) \leq \mathbf{b}(\mu_2)$ implies that $s_1 \geq s_2$. Since $\beta_1$ and $\beta_2$ both contain $k$ entries, the number of $2$'s in $\beta_1$ is smaller or equal to the number of $2$'s in $\beta_2$ (i.e., $v_1 \leq v_2$). Next, we assume the contrary that $u_1 > u_2$ (i.e., $\alpha_1$ contains more $2$'s than $\alpha_2$). If $s_1=s_2$, then $\mathbf{b}(\mu_1) > \mathbf{b}(\mu_2)$; if $s_1 > s_2$, then $\mathbf{b}(\mu_1)$ and $\mathbf{b}(\mu_2)$ are not comparable, both cases yield a contradiction. Hence, $u_1 \leq u_2$.
    
    It is straightforward to verify the converse, which we omit here.
\end{proof}

Consider a diagram consisting of $a$ left-aligned rows of unit squares, where the $i$th row contains $a+b-i$ squares. We denote this diagram by $S(a,b)$. If $b=0$, then the last row of the diagram $S(a,0)$ reduces to a single vertical edge. If $a=0$, then the diagram $S(0,b)$ reduces to a horizontal path of length $b$. If $a=b=0$, then $S(0,0)$ reduces to a single point. In the following theorem, we describe the lattice structure of the alt $2$-row-Tamari lattice $\T_{\delta(k)}(a,b)$.
\begin{theorem}\label{thm.lattice2row}
    Let $a,b$ be nonnegative integers.
    \begin{itemize}
        \item When $k=0$, the Hasse diagram of the alt $2$-row-Tamari lattice $\T_{\delta}(a,b)$ is the diagram $S(a,b)$ with an extra point connecting the top right vertex of $S(a,b)$; see Figure \ref{fig.2rowlattice1}.
        \item When $1 \leq k \leq b$, the Hasse diagram of the alt $2$-row-Tamari lattice $\T_{\delta}(a,b)$ is obtained from $S(a,b)$ by deforming the $k$th unit square, counted from right to left, in each row into a pentagon; see Figure \ref{fig.2rowlattice2}. 
    \end{itemize}
\end{theorem}
\begin{figure}[htbp]
    \centering
    \subfigure[]{\label{fig.2rowlattice1}\includegraphics[width=0.3\textwidth]{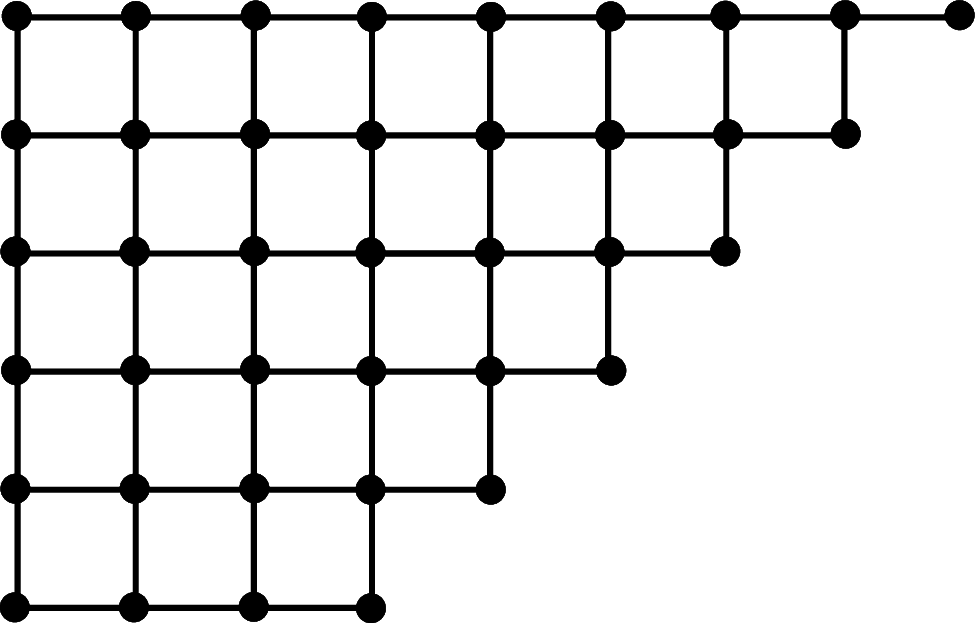}}
    \hspace{15mm}
    \subfigure[]{\label{fig.2rowlattice2}\includegraphics[width=0.3\textwidth]{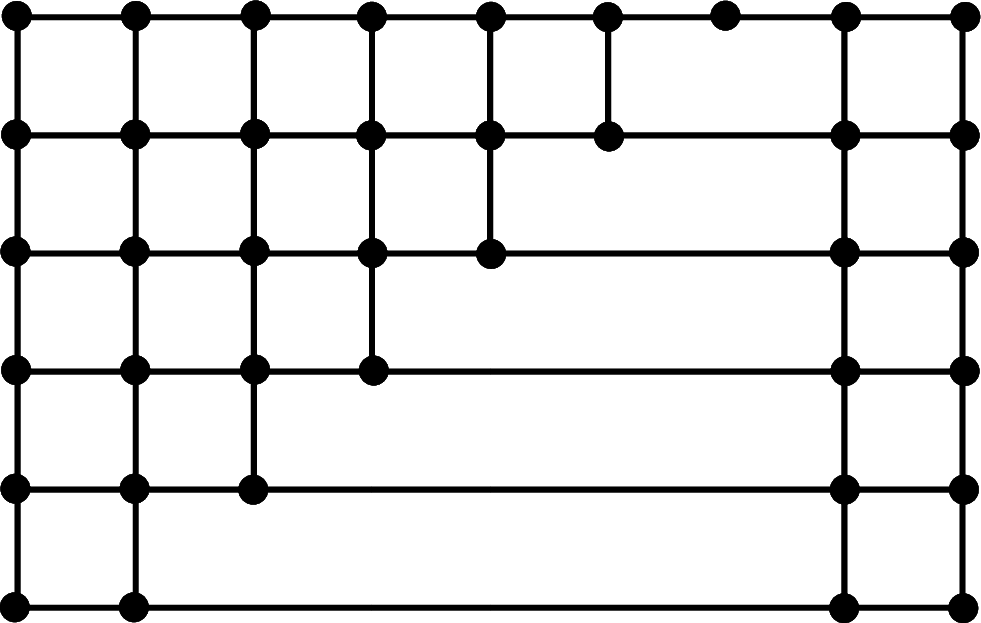}}
    \caption{The Hasse diagram of the alt $2$-row-Tamari lattice $\T_{\delta(k)}(a,b)$ with $(a,b)=(5,3)$ and (a) $k=0$, (b) $k=2$.}\label{fig.2rowlattice}
\end{figure}
\begin{proof}
    By Lemma \ref{lem.2rowbracket} and the discussion before Lemma \ref{lem.partialorder2row}, we may express each element $\mu$ of $\T_{\delta(k)}(a,b)$ as a triple $(a-s,u,v)$, where $s$ is the number of zeroes at the non-fixed positions of $\mathbf{b}(\mu)$ and $u$ (resp., $v$) is the number of $2$'s in $\alpha$ (resp., $\beta$). Recall that the possible pairs $(u,v)$ are given in \eqref{eq.2rowpartialorder}. 

    By Lemma \ref{lem.partialorder2row}, for a fixed value of $s=s_0$, the elements $\{(a-s_0,u,v)\}$ form the chain of $a+b-s_0+1$ elements, and then we place these elements on the line $y=a-s_0$ from left to right according to their partial order. Figure \ref{fig.2rowexample} shows the partial Hasse diagram of elements $\{(a-s_0,u,v)\}$ when $s=s_0,s_0-1,s_0-2$, and $v$ changes from $0$ to $1$. In the figure, we write $c=a+b-k-s_0$ for brevity. This is the place where a square is deformed into a pentagon in each row. 
    It is not hard to see that the remaining parts of the diagram are a rectangular lattice, which leads to the desired Hasse diagram as shown in Figure \ref{fig.2rowlattice2}.
    
    For $k=0$, $v$ is always $0$ in the triple; there is no square deformed into a pentagon. Hence, we obtain the desired Hasse diagram as shown in Figure \ref{fig.2rowlattice1}.
\end{proof}
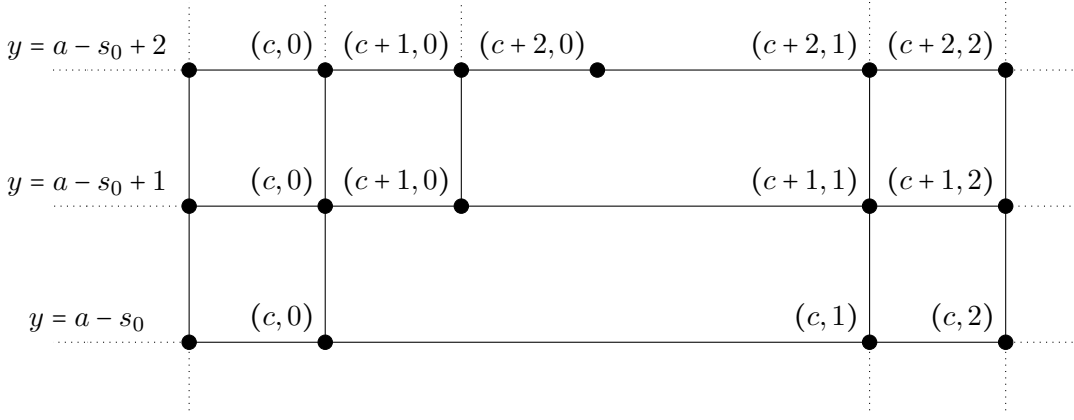
\begin{figure}[hbt!]
        \centering
        \begin{tikzpicture}[scale = 0.9]
            \draw (0,4) -- (12,4);
            \draw (0,2) -- (12,2);
            \draw (0,0) -- (12,0);
            \draw (0,0) -- (0,4);
            \draw (2,0) -- (2,4);
            \draw (4,2) -- (4,4);
            \draw (10,0) -- (10,4);
            \draw (12,0) -- (12,4);
            \draw[dotted] (10,5) -- (10,4);
            \draw[dotted] (10,-1) -- (10,0);
            \draw[dotted] (12,5) -- (12,4) -- (13,4);
            \draw[dotted] (12,2) -- (13,2);
            \draw[dotted] (12,-1) -- (12,0) -- (13,0);
            \draw[dotted] (0,-1) -- (0,0) -- (-2,0); 
            \filldraw[black] (-1.5,0) circle (0pt) node [anchor=south] {$y=a-s_0$};
            \filldraw[black] (-1.5,2) circle (0pt) node [anchor=south] {$y=a-s_0+1$};
            \filldraw[black] (-1.5,4) circle (0pt) node [anchor=south] {$y=a-s_0+2$};
            \draw[dotted] (0,2) -- (-2,2);
            \draw[dotted] (-2,4)--(0,4)--(0,5);
            \draw[dotted] (2,4)--(2,5);
            \draw[dotted] (4,4)--(4,5);
            \filldraw[black] (12,4) circle (3pt) node [anchor=south east] {$(c+2,2)$};
            \filldraw[black] (12,2) circle (3pt) node [anchor=south east] {$(c+1,2)$};
            \filldraw[black] (12,0) circle (3pt) node [anchor=south east] {$(c,2)$};
            \filldraw[black] (10,4) circle (3pt) node [anchor=south east] {$(c+2,1)$};
            \filldraw[black] (10,2) circle (3pt) node [anchor=south east] {$(c+1,1)$};
            \filldraw[black] (10,0) circle (3pt) node [anchor=south east] {$(c,1)$};
            \filldraw[black] (6,4) circle (3pt) node [anchor=south east] {$(c+2,0)$};
            \filldraw[black] (4,4) circle (3pt) node [anchor=south east] {$(c+1,0)$};
            \filldraw[black] (2,4) circle (3pt) node [anchor=south east] {$(c,0)$};
            \filldraw[black] (0,4) circle (3pt);
            \filldraw[black] (0,2) circle (3pt);
            \filldraw[black] (0,0) circle (3pt);
            \filldraw[black] (2,0) circle (3pt) node [anchor=south east] {$(c,0)$};
            \filldraw[black] (2,2) circle (3pt) node [anchor=south east] {$(c,0)$};
            \filldraw[black] (4,2) circle (3pt) node [anchor=south east] {$(c+1,0)$};
        \end{tikzpicture}
    \caption{An illustration of forming pentagons in the Hasse diagram of $\T_{\delta(k)}(a,b)$.}
    \label{fig.2rowexample}
\end{figure}

\subsection{The invariance of rowmotion on $\T_{\delta}(a,b)$}

In this section, we prove that the orbit structure of the alt $2$-row-Tamari lattice $\T_{\delta(k)}(a,b)$ under rowmotion is independent of $\delta(k)$. The statistic $\mathsf{ddeg}$ over each such orbit is independent of $\delta(k)$ as well. We begin with the following definition.

\begin{definition}\label{def.switch}
    Let $(L,\leq)$ be a semidistributive lattice. We say $L$ has an \textit{$n$-switching property} if $L$ satisfies the following conditions:
    \begin{itemize}
        \item[(1)] we can embed $L$ into the rectangular coordinate system and identify each element as a lattice point $(x,y)$, where $x,y \geq0$;
        \item[(2)] $(0,0),(0,1),\dots,(0,n) \in L$, $(0,m) \notin L$ for $m > n$, and $(0,i-1) \lessdot (0,i)$ for $i=1,2,\dots n$;
        \item[(3)] two elements $(c_1,d_1) \lessdot (c_2,d_2)$ imply that the pair of numbers $(c_2-c_1,d_2-d_1) = (0,1)$ or $(m,0)$ for some $m$;
        \item[(4)] two elements $(c_1,d) \leq (c_2,d)$ if and only if $c_1 \leq c_2$.
    \end{itemize}
\end{definition}
We remark that the Hasse diagrams shown in Figure \ref{fig.2rowlattice} have the $n$-switching property ($n=5$). 
We show in the following proposition that the right boundary of the Hasse diagram forms a staircase. 
\begin{proposition}
    Let $L$ be a semidistributive lattice with the $n$-switching property. Define $$x_{m}=\max \{i| (i,m) \in L\}$$ to be the $x$-coordinate of the rightmost element on $y=m$ of $L$, for $m=0,1,\dots,n$. Then $x_0\leq x_1 \leq \cdots \leq x_n$.
\end{proposition}
\begin{proof}
    We assume the contrary that $x_\ell > x_{\ell+1}$ for some $\ell$. If there exists an element, say $(c,d)$, covers $(x_{\ell},\ell)$, then by the property (3) in Definition \ref{def.switch}, either $(c-x_{\ell},d-\ell) = (0,1)$ or $(m,0)$ for some $m$. We have $(c,d)=(x_{\ell},\ell+1)$ (resp., $(x_{\ell}+m,\ell)$), which violates the assumption that the element $(x_{\ell+1},\ell+1)$ (resp., $(x_{\ell},\ell)$) is the rightmost element of $L$ on the line $y=\ell+1$ (resp., $y=\ell$). Thus, no element of $L$ covers $(x_{\ell},\ell)$. 
    
    This leads to the fact that the least upper bound of $(x_{\ell},\ell)$ and $(x_{\ell+1},\ell+1)$ does not exist, which implies that $L$ is not a lattice, a contradiction. 
\end{proof}

We define an operation on two semidistributive lattices having the $n$-switching property. 
\begin{definition}\label{def.star}
    Let $A$ and $B$ be two semidistributive lattices having the $n$-switching property. Define an operation $A * B$ to be the poset on the elements of the disjoint union $A \sqcup B$ with the cover relation $x \lessdot y$ if and only if one of the following conditions holds:
    \begin{itemize}
        \item[(1)] $x \lessdot y$ in $A$, or
        \item[(2)] $x \lessdot y$ in $B$, or
        \item[(3)] $x=(x_m,m)_{A} \in A$ and $y=(0,m)_{B} \in B$, for $m=0,1,\dots,n$. 
    \end{itemize}
\end{definition}
\begin{theorem}\label{thm.switchrow}
    Let $A$ and $B$ be two semidistributive lattices having the $n$-switching property. Then 
    \begin{itemize}
        \item[(1)] $A*B$ is semidistributive,
        \item[(2)] rowmotion on $A*B$ and $B*A$ have the same orbit structure, and
        \item[(3)] the statistic $\mathsf{ddeg}$ over each orbit of $A*B$ and $B*A$ is the same.
    \end{itemize}
\end{theorem}
\begin{proof}
    (1) We first show that $A*B$ is a lattice. Take $x,y \in A*B$, it is clear that $x$ and $y$ have a unique least upper bound and a unique greatest lower bound if $x,y \in A$ or $x,y \in B$. Without loss of generality, we assume $x=(x_1,x_2)_A \in A$ and $y=(y_1,y_2)_B \in B$. By Condition (4) in Definition \ref{def.switch} and Condition (3) in Definition \ref{def.star}, there exist $x^{\prime}=(0,x_2)_B \in B$ and $y^{\prime}=(0,y_2)_A \in A$ such that $x \leq x^{\prime}$ and $y^{\prime} \leq y$.  One can see that $x$ and $y$ have a unique greatest lower bound (resp., least upper bound), which is given by the greatest lower bound (resp., least upper bound) of $x$ and $y^{\prime}$ (resp., $x^{\prime}$ and $y$). 

    Second, to prove the semidistributivity, we need to show that for all $x,y \in A*B$ with $x \leq y$, the set $M=\{z \in A*B| z \wedge y=x\}$ has a unique maximal element and the set $N=\{z \in A*B|z \vee x =y\}$ has a unique minimal element (Section \ref{sec.row}). 
        
    We show that the set $M$ has a unique maximal element; the proof that $N$ has a unique minimal element is similar and is omitted here. If $x \leq y$ are both in $B$, then $M$ has a unique maximal element in $B$ (since $B$ is semidistributive) and also in $A*B$ (by Condition (3) in Definition \ref{def.star}, elements in $B$ are either greater than or incomparable with elements in $A$). If $x \leq y$ are both in $A$, then again $M$ has a unique maximal element $z$ in $A$ since $A$ is semidistributive. Among all the elements of $M$, if $z$ is also maximal in $A*B$, then we are done. If not, we define $z^{\prime}$ to be the rightmost element of $B$ with the same $y$-coordinate as $z$. Clearly, $z \leq z^{\prime}$ and the $y$-coordinate of $z^{\prime} \wedge y$ is less than or equal to the $y$-coordinate of $z^{\prime}$ and $z$. Thus, $z^{\prime} \wedge y  \leq z \leq z^{\prime}$, which implies that $z^{\prime} \wedge y = z \wedge y =x$. So, $z^{\prime}$ is the desired unique maximal element of $M$. 

    Finally, for the case when $x \in A$, $y=(y_1,y_2)_B \in B$, and $x \leq y$, we claim that the elements of $M$ must belong to $A$. Suppose $z=(z_1,z_2)_B \in M \cap B$, due to Condition (2) in Definition \ref{def.switch}, $z \wedge y \geq (0,\min \{y_2,z_2\})_B$. This implies that $x = z \wedge y \in B$, which violates that $x \in A$. Now, we assume $z \in M \cap A$. Define the element $y^{\prime}$ to be the rightmost element of $A$ with the same $y$-coordinate as $y$. Clearly, $y^{\prime} \leq y$ and the $y$-coordinate of $z \wedge y$ is less than or equal to the $y$-coordinate of $y^\prime$ and $y$. Thus, $z \wedge y \leq y^{\prime} \leq y$, implying that $z \wedge y^{\prime} = z \wedge y$. This case reduces to finding a unique maximal element in $\{z \in A | z \wedge y^{\prime}=x\}$, which was discussed in the previous paragraph. This completes the proof of the first part.

    (2) Due to part (1), $A*B$ and $B*A$ are semidistributive, thus we can apply rowmotion on these two lattices. Let $(a_m,m)_A$ (resp., $(b_m,m)_B$) denote the rightmost element of $A$ (resp., $B$) whose $y$-coordinate is $m$, for $0 \leq m \leq n$. The key point is to find out where the elements $L(A)=\{(0,m)_A|0 \leq m \leq n\}$ and $L(B)=\{(0,m)_B|0 \leq m \leq n\}$ map to under rowmotion in both $A*B$ and $B*A$. We have the following claim.\\

    \noindent \textbf{Claim.}
        Let $x \in A \setminus L(A)$ and $y \in B \setminus L(B)$. Then $\Row_{A*B}(x) = \Row_{B*A}(x) \in A$ and $\Row_{A*B}(y) = \Row_{B*A}(y) \in B$.
    \begin{proof}[Proof of Claim.]
        We assume the coordinate of $x$ is given by $x=(x_1,x_2)_A$ with $x_1 \neq 0$. We first find the pop-stack sorting operator $\Pop_{A*B}^{\downarrow}(x)=x \wedge \bigwedge \{z \in A*B | z \lessdot x\}$ and then apply Theorem \ref{thm.charrow}. By Condition (3) in Definition \ref{def.switch}, since $x \notin L(A)$, we have the following two cases:
        \begin{itemize}
            \item[Case 1.] The element $x=(x_1,x_2)_A$ covers $c=(x_1,x_2-1)_A$ and $d=(x_1-m,x_2)_A$ in $A$, for some $m$ (the element $c$ always exists in $A$ due to Condition (2) in Definition \ref{def.switch}). Then $\Pop_{A*B}^{\downarrow}(x) = \Pop_{B*A}^{\downarrow}(x) = x \wedge c \wedge d$, which is the minimal element of the face containing $x$ as the maximal element. In this case, $\Row_{A*B}(x) = \Row_{B*A}(x) = c \wedge d \in A$. 
            
            \item[Case 2.] The element $x=(x_1,x_2)_A$ covers only $c=(x_1-m,x_2)_A$ for some $m$. Then $\Pop_{A*B}^{\downarrow}(x) = \Pop_{B*A}^{\downarrow}(x) = x \wedge c = c$. We define $R=\{z \in A*B | z \wedge x = c \}$, our goal is to show that $\Row_{A*B}(x) = \max R \in A$. Take an arbitrary element $w=(w_1,w_2)_B \in B$, if $w_2 \geq x_2$, then $x \leq (0,x_2)_B \leq (0,w_2)_B \leq w$ which implies that $w \wedge x = x$, so $w \notin R$. On the other hand, if $w_2 < x_2$, then the $y$-coordinate of $w \wedge x$ is $w_2$ which is less than the $y$-coordinate of $c$, so $w \notin R$. Thus, the element $z \in A*B$ having the property that $z \wedge x =c$ must lie in $A$. We also notice that $\Row_{B*A}(x) \wedge x = c$, then $c \leq \Row_{B*A}(x)$ which implies that $\Row_{B*A}(x) \in A$. Rowmotion on $A*B$ and $B*A$ is restricted to $A$, we have $\Row_{A*B}(x) = \Row_{B*A}(x) \in A$.
        \end{itemize}
        We can similarly argue that the image of an element of $B \setminus L(B)$ lies in $B$ under rowmotion, which is omitted here. 
    \end{proof}

   Next, we determine the image of elements of $L(A)$ and $L(B)$ under rowmotion on $A*B$. Notice that
    \begin{small}
    \begin{equation}
        \Pop_{A*B}^{\downarrow}((0,m)_A)=\begin{cases}
            (0,m-1)_A, & \text{if $m \neq 0$,}\\
            (0,0)_A, & \text{if $m=0$,}
        \end{cases}
        \quad \text{and} \quad
        \Pop_{A*B}^{\downarrow}((0,m)_B)=\begin{cases}
            (a_{m-1},m-1)_A, & \text{if $m \neq 0$,}\\
            (a_0,0)_A, & \text{if $m=0$.}
        \end{cases}
    \end{equation}
    \end{small}
    By Theorem \ref{thm.charrow}, one can readily check that
    \begin{small}
    \begin{equation}\label{eq.AstarB}
        \Row_{A*B}((0,m)_A) = \begin{cases}
            (b_{m-1},m-1)_B, & \text{if $m \neq 0$,}\\
            (b_n,n)_B, & \text{if $m=0$,}
        \end{cases}
        \quad \text{and} \quad 
        \Row_{A*B}((0,m)_B)=\begin{cases}
            (a_{m-1},m-1)_A, & \text{if $m \neq 0$,}\\
            (a_n,n)_A, & \text{if $m=0$.}
        \end{cases}
    \end{equation}
    \end{small}
    Similarly, in $B*A$, we have 
    \begin{small}
    \begin{equation}
        \Pop_{B*A}^{\downarrow}((0,m)_B)=\begin{cases}
            (0,m-1)_B, & \text{if $m \neq 0$,}\\
            (0,0)_B, & \text{if $m=0$,}
        \end{cases}
        \quad \text{and} \quad
        \Pop_{B*A}^{\downarrow}((0,m)_A)=\begin{cases}
            (b_{m-1},m-1)_B, & \text{if $m \neq 0$,}\\
            (b_0,0)_B, & \text{if $m=0$.}
        \end{cases}
    \end{equation}
    \end{small}
    By Theorem \ref{thm.charrow}, one can readily check that
    \begin{small}
    \begin{equation}\label{eq.BstarA}
        \Row_{B*A}((0,m)_A) = \begin{cases}
            (b_{m-1},m-1)_B, & \text{if $m \neq 0$,}\\
            (b_n,n)_B, & \text{if $m=0$,}
        \end{cases}
        \quad \text{and} \quad 
        \Row_{B*A}((0,m)_B)=\begin{cases}
            (a_{m-1},m-1)_A, & \text{if $m \neq 0$,}\\
            (a_n,n)_A, & \text{if $m=0$.}
        \end{cases}
    \end{equation}
    \end{small}

    By Claim, \eqref{eq.AstarB} and \eqref{eq.BstarA}, rowmotion on $A*B$ and rowmotion on $B*A$ behave the same. Therefore, rowmotion on $A*B$ and rowmotion on $B*A$ have the same orbit structure. This proves the second part.

    (3) From the proof of part (2), we know that the elements of the orbit of $A*B$ are the same as the elements of the corresponding orbit of $B*A$. In $A*B$ and $B*A$, the statistic $\mathsf{ddeg}$ only differs at the elements $(0,m)_A \in L(A)$ and $(0,m)_B \in L(B)$ for $m=0,1,\dots,n$. In $A*B$,
    \begin{equation}
        \mathsf{ddeg}(0,m)_A = \begin{cases}
            1, \text{ if $m \neq 0$,}\\
            0, \text{ if $m=0$,}
        \end{cases}
        \quad \text{and} \quad 
        \mathsf{ddeg}(0,m)_B = \begin{cases}
            2, \text{ if $m \neq 0$,}\\
            1, \text{ if $m=0$.}
        \end{cases}
    \end{equation}
    On the other hand, in $B*A$, 
    \begin{equation}
        \mathsf{ddeg}(0,m)_B = \begin{cases}
            1, \text{ if $m \neq 0$,}\\
            0, \text{ if $m=0$,}
        \end{cases}
        \quad \text{and} \quad 
        \mathsf{ddeg}(0,m)_A = \begin{cases}
            2, \text{ if $m \neq 0$,}\\
            1, \text{ if $m=0$.}
        \end{cases}
    \end{equation}
    
    In each orbit (of $A*B$ or $B*A$), the number of elements of $L(A)$ is the same as the number of elements of $L(B)$. This is due to Claim, \eqref{eq.AstarB} and \eqref{eq.BstarA} in the proof of the second part, rowmotion sends an element of $L(A)$ (resp., $L(B)$) to an element in $B$ (resp., $A$) while sending an element of $A \setminus L(A)$ (resp., $B \setminus L(B)$) to an element of $A$ (resp., $B$). Thus, each orbit contains the same number of elements of $L(A)$ and $L(B)$. This implies that the statistic $\mathsf{ddeg}$ over each orbit of $A*B$ and $B*A$ is the same. This finishes the proof of Theorem \ref{thm.switchrow}.
\end{proof}

\begin{theorem}\label{thm.2roworbit}
    Let $a,b$, and $k$ be nonnegative integers with $0 \leq k \leq b$. Under rowmotion on the alt $2$-row-Tamari lattice $\T_{\delta(k)}(a,b)$, the following two quantities are independent of the choices of the increment vector $\delta(k)$:
    \begin{itemize}
        \item[(1)] the orbit structure, and 
        \item[(2)] the statistic $\mathsf{ddeg}$ over each orbit. 
    \end{itemize}
\end{theorem}
\begin{proof}
    Recall that the Hasse diagram of $\T_{\delta(k)}(a,b)$ is presented in Theorem \ref{thm.lattice2row}. $\T_{\delta(0)}(a,b)$ can be viewed as $A_k*B_k$, where $A_k$ is the product of two chains $C_{a+1} \times C_k$ and $B_k$ is $\T_{\delta(0)}(a,b-k)$, for each $k=1,2,\dots,b$. We note that both $A_k$ and $B_k$ are semidistributive lattices with the $a$-switching property. It is obvious that $\T_{\delta(k)}(a,b)$ is given by $B_k*A_k$. Then the first result follows from  Theorem \ref{thm.switchrow} (2). The second result follows from Theorem \ref{thm.switchrow} (3).
\end{proof}

\subsection{The orbit structure of $\T_{\delta}(a,b)$ under rowmotion}

In this section, we give an explicit orbit structure of the alt $2$-row-Tamari lattice $\T_{\delta(k)}(a,b)$ under rowmotion. We also show that the down-degree statistic on $\T_{\delta(k)}(a,b)$ is homometric under rowmotion. Due to Theorem \ref{thm.2roworbit}, it suffices to analyze the orbit structure of $\T_{\delta(k)}(a,b)$ under rowmotion when $\delta(k)=(0,0)$, that is, the $\nu$-Dyck lattice with $\nu = E^aNE^bN$. The statement is presented in the following theorem.

\begin{theorem}\label{thm.orbit}
    Let $a$ and $b$ be nonnegative integers. If $b=s(a+1)+r$ for $s\in \mathbb{Z}_{\geq 0}$ and $0 \leq r \leq a$, then the orbit structure and the down-degree statistic of the alt $2$-row-Tamari lattice $\T_{\delta(k)}(a,b)$ under rowmotion are independent of $\delta(k)$ and has the following properties:
    \begin{itemize}
        \item[(1)] for the orbit structure,
        \item there are $\displaystyle \left\lceil  \frac{a-r}{2} \right\rceil$ orbits $\mathcal{O}_1$ of size $a+2b+2-r$,
        \item there are $\displaystyle \left\lceil  \frac{r-1}{2} \right\rceil$ orbits $\mathcal{O}_2$ of size $2a+2b+3-r$,
        \item if $a$ and $r$ are even, then there is one extra orbit $\mathcal{O}_3$ of size $\displaystyle \frac{a-r}{2}+b+1$, 
        \item if $a$ is even and $r$ is odd, then there is one extra orbit $\mathcal{O}_4$ of size $\displaystyle \frac{3-r}{2}+a+b$, 
        \item if $a$ and $r$ are odd, then there are two extra orbits $\mathcal{O}_3$ and $\mathcal{O}_4$ of sizes $\displaystyle \frac{a-r}{2}+b+1$ and $\displaystyle \frac{3-r}{2}+a+b$, respectively. 

        \item[(2)] for the down-degree statistic,
            \begin{align*}
                \sum_{z \in \mathcal{O}_1} \mathsf{ddeg}(z) & = 2 |\mathcal{O}_1| -4-2s, \quad 
                \sum_{z \in \mathcal{O}_2} \mathsf{ddeg}(z)  = 2 |\mathcal{O}_2| -6-2s, \\
                \sum_{z \in \mathcal{O}_3} \mathsf{ddeg}(z) & = 2 |\mathcal{O}_3| -2-s, \quad \text{and} \quad
                \sum_{z \in \mathcal{O}_4} \mathsf{ddeg}(z)  = 2 |\mathcal{O}_4| -3- s.
            \end{align*}
        Moreover, the statistic $\mathsf{ddeg}$ is homometric under rowmotion.
    \end{itemize}
\end{theorem}

The proof of Theorem \ref{thm.orbit} is outlined as follows:
\begin{itemize}
    \item For $a+b<2$, we analyze the orbit structure separately as their lattice structures are simple. This will be discussed at the beginning of the proof of Theorem \ref{thm.orbit}.
    \item For $a+b \geq 2$, we turn the problem into solving a linear congruence equation (Lemmas \ref{lem.2roworbit1} and \ref{lem.2roworbit2}).
    \item We further narrow down to the case $0 \leq b \leq a$ for solving the linear congruence equations \eqref{eq.equivb0} and \eqref{eq.equivb1}, which is presented in Lemma \ref{lem.2rowsol}.
    \item We finally present the proof of Theorem \ref{thm.orbit}.
\end{itemize}

We assume $a+b \geq 2$. Following Theorem \ref{thm.charrow}, one can readily check that rowmotion on $\mathsf{Dyck}(\nu)$ with $\nu=E^aNE^bN$ is given by
\begin{equation}\label{eq.rowmotionDyck2row}
    \Row(x,y)=\begin{cases}
        (a+b,a) & \text{if $x=y=0$},\\
        (y+b-1,y-1) & \text{if $x=0$ and $y>0$},\\
        (x-1,a) & \text{if $y=0$ and $x>0$, or $x=y+b$},\\
        (x-1,y-1) & \text{otherwise}.
    \end{cases}
\end{equation}
We observe two facts: first, rowmotion always maps the maximal element $(a+b,a)$ to $(a+b-1,a)$. Second, for any element $(x,y)$, there exists $k \in \mathbb{Z}_{\geq 0}$ such that $\Row^k(x,y) = (x^{\prime},a)$ for some $x^{\prime}$. The key idea is to find out which elements $(x,a)$, $0 \leq x \leq a+b-1$ (we exclude the maximal element) belong to the same orbit under rowmotion. Two lemmas are given below.
\begin{lemma}\label{lem.2roworbit1}
    Let $0 \leq x_1,x_2 \leq a+b-1$. When $b=0$, two elements $(x_1,a)$ and $(x_2,a)$ are in the same orbit if and only if
    \begin{equation}\label{eq.equivb0}
        x_1+x_2 \equiv a-2 \pmod{a+1} \quad \text{or} \quad  x_1=x_2=a-1.
    \end{equation}
    When $b>0$, two elements $(x_1,a)$ and $(x_2,a)$ are in the same orbit if and only if 
    \begin{equation}\label{eq.equivb1}
        x_1 + x_2 \equiv a+b-2 \pmod{a+1}\quad \text{or} \quad x_1 \equiv x_2 \pmod{a+1}.
    \end{equation}
\end{lemma}
\begin{proof}
    Assuming $b>0$, we discuss the following cases based on the value of $x$.
    \begin{itemize}
        \item[(Case 1)] $a+1 \leq x \leq a+b-1$. By \eqref{eq.rowmotionDyck2row}, $\Row^a(x,a) = (x-a,0)$ and $\Row^{a+1}(x,a) = (x-(a+1),a)$. 

        \item[(Case 2)] $x=a$. By \eqref{eq.rowmotionDyck2row}, $\Row^a(x,a) = (0,0)$, $\Row^{a+1}(x,a) = (a+b,a)$, and $\Row^{a+2}(x,a) = (a+b-1,a)$. Note that $2a+b-1 \equiv a+b-2 \pmod{a+1}$. 

        \item[(Case 3)] $0 \leq x \leq a-1$. By \eqref{eq.rowmotionDyck2row}, $\Row^x(x,a) = (0,a-x)$, $\Row^{x+1}(x,a) = (a-x+b-1,a-x-1)$, and $\Row^{x+2}(x,a) = (a-x+b-2,a)$. 
    \end{itemize}
    Case 1 leads to the equation $x_1 \equiv x_2 \pmod{a+1}$ while Cases 2 and 3 provide the equation $x_1+x_2 \equiv a+b-2 \pmod{a+1}$. 
    
    When $b=0$, we only have Case 3. Moreover, if $x=a-1$, then $\Row^{a}(a-1,a) = (0,0)$, $\Row^{a+1}(a-1,a)=(a,a)$, and $\Row^{a+2}(a-1,a)=(a-1,a)$. This shows that $(x_1,a)$ and $(x_2,a)$ are in the same orbit if and only if $x_1+x_2 \equiv a-2 \pmod{a+1}$ or $x_1=x_2=a-1$.
\end{proof}

\begin{lemma}\label{lem.2roworbit2}
    Each orbit contains either one or two elements whose coordinates are given by $(x,a)$ for $\varepsilon(b>0) \cdot (b-1) \leq x \leq a+b-1$. 
\end{lemma}
\begin{proof}
    We assume that there are more than two elements $(x_i,a)$, $b-1 \leq x_i \leq a+b-1$ in the same orbit. We apply Lemma \ref{lem.2roworbit1}. Within the chosen interval of length $a$,  distinct elements must satisfy $x_{i_1}+x_{i_2} \equiv a+b-2 \pmod{a+1}$ and $x_{i_2}+x_{i_3} \equiv a+b-2 \pmod{a+1}$. This implies that $x_{i_3}-x_{i_1} \equiv 0 \pmod{a+1}$. Since $-a \leq x_{i_3}-x_{i_1} \leq a$, this forces $x_{i_3}=x_{i_1}$. When $b=0$, we require $x_i$'s to be nonnegative; the above argument still holds. So, each orbit contains at most two elements $(x_i,a)$ where $\varepsilon(b>0) \cdot (b-1) \leq x_i \leq a+b-1$.
\end{proof}

Solving the linear congruence equations \eqref{eq.equivb0} and \eqref{eq.equivb1} is crucial for analyzing the orbit structure. By Lemma \ref{lem.2roworbit2}, each orbit contains at most two elements $(x,a)$ in the representative interval $\varepsilon(b>0) \cdot (b-1) \leq x \leq a+b-1$. When $b>a$, we can shift the interval to $r-1 \leq x \leq a+r-1$, where $r \equiv b \pmod{a+1}$. Thus, we may narrow down to the case $0 \leq b \leq a$. To describe the solutions of \eqref{eq.equivb0} and \eqref{eq.equivb1}, we define the following sets, which consist of pairs of nonnegative integers. 
\begin{align*}
    X_1 & = \left\{ \left( \frac{a}{2}-1-m,\frac{a}{2}-1+m  \right) \bigg| m=0,1,\dots,\frac{a-2}{2} \right\} \quad \text{for $a$ even},\\
    X_2 & = \left\{ \left( \frac{a}{2}-1-\frac{2m+1}{2},\frac{a}{2}-1+\frac{2m+1}{2}  \right) \bigg| m=0,1,\dots,\frac{a-3}{2} \right\} \quad \text{for $a$ odd},\\
    Y_1 &= \left\{ \left(a+\frac{b-1}{2}-m,a+\frac{b-1}{2}+m \right) \bigg| m=0,1,\dots,\frac{b-1}{2} \right\} \quad \text{for $b$ odd}, \\
    Y_2 &= \left\{ \left(a+\frac{b-1}{2}-\frac{2m+1}{2},a+\frac{b-1}{2}+\frac{2m+1}{2} \right) \bigg| m=0,1,\dots,\frac{b-2}{2} \right\} \quad \text{for $b$ even}, \\
    Z_1 &= \left\{ \left(\frac{a+b}{2}-1-m,\frac{a+b}{2}-1+m \right) \bigg| m=0,1,\dots,\frac{a-b}{2} \right\} \quad \text{for $a \equiv b \pmod{2}$}, \text{ and}\\
    Z_2 &= \left\{ \left( \frac{a+b}{2}-1-\frac{2m+1}{2},\frac{a+b}{2}-1+\frac{2m+1}{2} \right) \bigg| m=0,1,\dots,\frac{a-b-1}{2} \right\} \quad \text{for $a \not \equiv b \pmod{2}$}.
\end{align*}
\begin{lemma}\label{lem.2rowsol}
    Assume $0 \leq b \leq a$. When $b>0$, the solution set $X(a,b)=\{(x_1,x_2) | \varepsilon(b>0) \cdot (b-1) \leq x_1,x_2 \leq a+b-1\}$ of 
    \begin{equation}\label{eq.congurence}
        x_1 + x_2 \equiv a+b-2 \pmod{a+1}
    \end{equation}
    is given by
    \begin{equation*}
        X(a,b) = \begin{cases}
            Y_1 \cup Z_1 & \text{for $a,b$ odd,}\\
            Y_2 \cup Z_1 & \text{for $a,b$ even,}\\
            Y_1 \cup Z_2 & \text{for $a$ even and $b$ odd,}\\
            Y_2 \cup Z_2 & \text{for $a$ odd and $b$ even.}\\
        \end{cases}
    \end{equation*}
    When $b=0$, the solution set reduces to
    \begin{equation*}
        X(a,0) = \begin{cases}
            X_1 & \text{for $a$ even,}\\
            X_2 & \text{for $a$ odd.}
        \end{cases}
    \end{equation*}
\end{lemma}
\begin{proof}
    For $a,b$ odd and $b>0$, it is straightforward to verify that each pair $(x_1,x_2)$ in $Y_1 \cup Z_1$ satisfies \eqref{eq.congurence} and these numbers range from $b-1$ to $a+b-1$ exactly once except for two numbers $a+\frac{b-1}{2}$ and $\frac{a+b}{2}-1$ (when $m=0$). The remaining three cases can be verified in a similar way, which are left to the reader.
    
    For $b=0$, since $x_1,x_2$ are nonnegative, we delete one pair from the sets $Z_1$ (when $m=\frac{a}{2}$) and $Z_2$ (when $m = \frac{a-1}{2}$) involving $-1$, which results in the sets $X_1$ and $X_2$, respectively. It is obvious that $Y_2$ is empty. Thus, the solution sets reduce to $X_1$ for $a$ even and $X_2$ for $a$ odd.
\end{proof}

Now, we present the proof of Theorem \ref{thm.orbit}.
\begin{proof}[Proof of Theorem \ref{thm.orbit}]
    Due to Theorem \ref{thm.2roworbit}, the orbit structure and the down-degree statistic are independent of the increment vector $\delta(k)$. We take $\delta(k)=(0,0)$.
    
    When $a+b<2$, there are three cases $(a,b)=(0,0),(0,1)$, and $(1,0)$. Respectively, their Hasse diagrams reduce to a singleton, a chain of two elements, and a chain of three elements. Under rowmotion, this gives one orbit of sizes $1,2$, and $3$. Moreover, the total down-degree statistic over the corresponding orbits is $0,1$, and $2$. It is easy to verify that this agrees with the statement. 

    Now, we assume $a+b \geq 2$ and consider the case where $0 \leq b \leq a$. By Lemmas \ref{lem.2roworbit1} and \ref{lem.2roworbit2}, the linear congruence equations \eqref{eq.equivb0} and \eqref{eq.equivb1} have a solution $(x_1,x_2)$ ($\varepsilon(b>0) \cdot (b-1) \leq x_1,x_2 \leq a+b-1$) if and only if the orbit consists of two elements $(x_1,a)$ and $(x_2,a)$. Notice that when successively applying rowmotion on $(x_1,a)$, each rowmotion decreases the $x$-coordinate by $1$ (by \eqref{eq.rowmotionDyck2row}) until we reach the $y$-axis, this takes $x_1$ applications of rowmotion. From here, we need $2$ more applications of rowmotion to reach $(x_2,a)$. If $x_1=x_2$, then we obtain an orbit with size $x_1+2$. If $x_1 \neq x_2$, then by a similar reasoning, we require $x_2+2$ applications of rowmotion on $(x_2,a)$ to reach $(x_1,a)$. In total, this gives an orbit with size $x_1+x_2+4$.

    Next, we determine the number of orbits and their sizes under rowmotion, which can be obtained by counting the number of solutions $X(a,b)$ given in Lemma \ref{lem.2rowsol}. Consider the case when $b>0$, the sets $Y_1$ and $Y_2$ both consist of $\left \lceil \frac{b-1}{2} \right\rceil$ pairs $(x_1,x_2)$ with $x_1 \neq x_2$, which implies that the orbit containing these two distinct elements $(x_1,a)$ and $(x_2,a)$ has size $x_1+x_2+4 = 2a+b+3$. Moreover, $Y_1$ consists of one extra pair $(x_1,x_2)$ with $x_1=x_2$, which gives one orbit of size $x_1+2 = a + \frac{b+3}{2}$ (when $b$ is odd). We also notice that from $(x_1,a)$ to $(x_2,a)$ in these orbits, the sequence of applications of rowmotion involves three elements, which are on the left, bottom, and right sides of its Hasse diagram, respectively. It is also possible to involve two elements, the element $(0,0)$ and the maximal element. In both situations, the total down-degree statistic over the elements of each orbit is given by twice the orbit size minus $3$ (if $x_1=x_2$) or minus $6$ (if $x_1 \neq x_2$). 
    
    Similarly, the sets $Z_1$ and $Z_2$ both consist of $\left \lceil \frac{a-b}{2} \right\rceil$ pairs $(x_1,x_2)$ with $x_1 \neq x_2$, which implies that the orbit containing these two distinct elements $(x_1,a)$ and $(x_2,a)$ has size $x_1+x_2+4 = a+b+2$. Moreover, $Z_1$ consists of one extra pair $(x_1,x_2)$ with $x_1=x_2$, which gives one orbit of size $x_1+2 = \frac{a+b}{2}+1$ (when $a \equiv b \pmod{2}$). Note that from $(x_1,a)$ to $(x_2,a)$ in these orbits, the sequence of applications of rowmotion involves two elements (excluding $(0,0)$), which are on the left and right sides of its Hasse diagram, respectively. Therefore, the total down-degree statistic over the elements of each orbit is given by twice the orbit size minus $2$ (if $x_1=x_2$) or minus $4$ (if $x_1 \neq x_2$).

    For the case when $b=0$, the sets $X_1$ and $X_2$ consist of $\left \lceil \frac{a-2}{2} \right\rceil$ pairs $(x_1,x_2)$ with $x_1 \neq x_2$. This implies that the orbit containing these two distinct elements $(x_1,a)$ and $(x_2,a)$ has size $x_1+x_2+4 = a+2$. Moreover, $X_1$ consists of one extra pair $(x_1,x_2)$ with $x_1=x_2$, which gives one orbit of size $x_1+2 = \frac{a}{2}+1$ (when $a$ is even). Notice that there is one more orbit which contains the element $(a-1,a)$ (from \eqref{eq.equivb0}), it has size $a+2$. By a similar argument in the previous paragraph, the total down-degree statistic over the elements of each orbit is given by twice the orbit size minus $2$ (if $x_1=x_2$) or minus $4$ (if $x_1 \neq x_2$).

    We summarize the above discussion (when $a+b \geq 2$ and $0 \leq b \leq a$) as follows:
    \begin{itemize}
        \item for any $a$ and $b$, there are $\left \lceil \frac{a-b}{2} \right\rceil$ orbits of size $a+b+2$,
        \item for any $a$ and $b$, there are $\left \lceil \frac{b-1}{2} \right\rceil$ orbits of size $2a+b+3$,
        \item when $a$ and $b$ are even, there is one extra orbit of size $\frac{a+b}{2}+1$,
        \item when $a$ is even and $b$ is odd, there is one extra orbit of size $a+\frac{b+3}{2}$, and
        \item when $a$ and $b$ are both odd, there are two extra orbits of sizes $a+\frac{b+3}{2}$ and $\frac{a+b}{2}+1$.
    \end{itemize}
    
    Finally, we consider the case $b > a$. Let $r$ be the remainder and $s$ be the quotient of $b$ when dividing by $a+1$. The Hasse diagram of $\T_{\delta(k)}(a,b)$ can be obtained from the Hasse diagram of $\T_{\delta(k)}(a,r)$ by attaching a rectangular lattice of size $a \times (b-r)$ to its left. The number of orbits can be determined by analyzing rowmotion on $\T_{\delta(k)}(a,r)$, namely the solution set $X(a,r)$, as summarized in the previous paragraph. Note that each element $(x_1,a)$ requires extra $s(a+1) = b-r$ applications of rowmotion to reach $(x_2,a)$, where $b-1 \leq x_1,x_2 \leq a+b-1$. This gives the orbit sizes of rowmotion on $\T_{\delta(k)}(a,b)$. To recover the total down-degree statistic over the elements in each orbit, we need to count the number of elements on the bottom side of its Hasse diagram in each orbit. It is not hard to see that from $(x_1,a)$ to $(x_2,a)$ in each orbit, the sequence of applications of rowmotion involves $\left \lfloor \frac{b}{a+1} \right \rfloor =s$ elements on the bottom side of its Hasse diagram. 

    For any $a$ and $b$, there are $\left \lceil \frac{a-r}{2} \right\rceil$ orbits $\mathcal{O}_1$ of size $a+r+2+2(b-r) = a+2b+2-r$, and
    \begin{equation*}
        \sum_{z \in \mathcal{O}_1} \mathsf{ddeg}(z) = 2|\mathcal{O}_1| - 4 - 2s.
    \end{equation*}
    There are also $\left \lceil \frac{r-1}{2} \right\rceil$ orbits $\mathcal{O}_2$ of size $2a+r+3+2(b-r) = 2a+2b+3-r$, and 
    \begin{equation*}
        \sum_{z \in \mathcal{O}_2} \mathsf{ddeg}(z) = 2|\mathcal{O}_2| - 6 - 2s.
    \end{equation*}
    When $a$ and $r$ are even, there is one extra orbit $\mathcal{O}_3$ of size $\frac{a+r}{2}+1+(b-r) = \frac{a-r}{2}+b+1$, and
    \begin{equation*}
        \sum_{z \in \mathcal{O}_3} \mathsf{ddeg}(z) = 2|\mathcal{O}_3| - 2 -s.
    \end{equation*}
    When $a$ is even and $r$ is odd, there is one extra orbit $\mathcal{O}_4$ of size $a+\frac{r+3}{2}+(b-r) = \frac{3-r}{2}+a+b$, and
    \begin{equation*}
        \sum_{z \in \mathcal{O}_4} \mathsf{ddeg}(z) = 2|\mathcal{O}_4| - 3 - s.
    \end{equation*}
    Lastly, when $a$ and $r$ are both odd, there are two extra orbits $\mathcal{O}_3$ and $\mathcal{O}_4$, as desired. Since the total value $\mathsf{ddeg}$ only depends on $s$ (the quotient of $b$ when dividing by $a+1$) and the size of the orbit, it is homometric under rowmotion. This completes the proof of Theorem \ref{thm.orbit}.
\end{proof}

\section{Concluding remarks}\label{sec.remarks}

In this paper, we study rowmotion on two families of alt $\nu$-Tamari lattices: the alt hook-Tamari lattice $\H_{\delta}(a,b)$, corresponding to the path $\nu=EN^{a-1}E^{b-1}N$; and the alt $2$-row-Tamari lattice $\T_{\delta}(a,b)$, corresponding to the path $\nu=E^aNE^bN$. Our results show that these two families support Conjecture \ref{conj}, which states that the rowmotion orbit structure of an alt $\nu$-Tamari lattice depends only on the path $\nu$ and is independent of the increment vector $\delta$.  Furthermore, these two families also support Conjecture \ref{conjhomometry} regarding the homometry of the down-degree statistic. It is natural to extend Conjecture \ref{conjhomometry} from $\nu$-Tamari lattices to alt $\nu$-Tamari lattices, the statement is given below:
\begin{conjecture}
    Let $\nu$ be a lattice path. The down-degree statistic is homometric for rowmotion on alt $\nu$-Tamari lattices $\Tam_{\delta}(\nu)$ and is independent of the increment vector $\delta$. 
\end{conjecture}

It would be interesting to find a direct explanation of the invariance of the rowmotion orbit structure and the orbit sum of the down-degree statistic for general alt $\nu$-Tamari lattices. Another possible direction is to study whether the switching property (Theorem \ref{thm.switchrow}) used in the $2$-row case can be applied to other families of semidistributive lattices.

\noindent \textbf{Acknowledgements.} 

This research was supported by National Science and Technology Council, Taiwan, through grants 113-2115-M-003-010-MY3 (S.-P. Eu) and 114-2811-M-003-024 (Y.-L. Lee).




\end{document}